\documentclass[a4paper, reqno]{amsart}

\usepackage{color}
\usepackage[T1]{fontenc}
\usepackage{amsfonts}
\usepackage{amssymb}
\usepackage{amsmath}
\usepackage{amsthm}
\usepackage{amsfonts}  % blackboard math symbols
\usepackage{dsfont,mathrsfs}
\usepackage{stmaryrd} % cope with double bar space
\usepackage{enumitem}
\usepackage{amscd}
\usepackage{pgf}
\usepackage{tikz}
\usepackage{hyperref}
\usetikzlibrary{matrix}
\usepackage{mathtools}
\usepackage{microtype}
\usepackage{etoolbox}

\usepackage{todonotes}

\newtheorem{theorem}{Theorem}[section]
\newtheorem{lemma}[theorem]{Lemma}
\newtheorem{proposition}[theorem]{Proposition}
\newtheorem{corollary}[theorem]{Corollary}

\theoremstyle{definition}
\newtheorem{definition}{Definition}[section]
\newtheorem{example}[definition]{Example}
\newtheorem{remark}[definition]{Remark}

\AfterEndEnvironment{definition}{\noindent\ignorespaces}
\AfterEndEnvironment{example}{\noindent\ignorespaces}
\AfterEndEnvironment{remark}{\noindent\ignorespaces}

\newcommand{\indicator}{\mathds{1}}
\newcommand{\ind}[1]{\mathds{1}_{\{ #1 \}}}

\newcommand{\g}[1]{\mathcal{G}_{#1}} % style G
\newcommand{\dom}[1]{\operatorname{Dom}{#1}} % domaine de la  divergence
\newcommand{\var}{\operatorname{var}}
\newcommand{\excond}[2]{\mathbb{E}\left[ \left.#1 \hspace{.2em}\right|\hspace{.2em} #2\right]} % conditional expectation
\newcommand{\cov}{\operatorname{cov}}

\newcommand{\Normal}{\mathcal{N}} % Normal
\newcommand{\bigCI}{\mathrel{\text{\scalebox{1.07}{$\perp\mkern-10mu\perp$}}}}

\def\P{\mathbf P}
\def\E{\mathbb{E}}
\def\I{\mathcal{J}} %sous-ensembles de {1...N}

\def\L{\mathsf{L}}
\def\deltaa{\Delta^{\{a\}}}
\def\deltaap{\Delta^{\{a\}'}}
\def\espz#1{\mathbb E\left[#1 | Z \right]}

\def\espxz#1{\mathbb E\left[#1 | X, Z \right]}
\def\cyl{\mathcal S}% Cylindrical functionals
\DeclareMathOperator{\oD}{d}
\def\D{\oD\!}
\def\dif{\,\D} % element differential
\def\sminus{\setminus}

\newcommand{\R}{{\mathbb R}}
\newcommand{\N}{{\mathbb N}}

\begin{document}
\title{Malliavin structure for conditionally independent random variables}
% \title{Dirichlet-Malliavin calculus for conditional limit theorems}

\author{L. Decreusefond \and C. Vuong}
% \author[ltci]{Laurent Decreusefond} \ead{Laurent.Decreusefond@mines-telecom.fr}

% \author[ltci,esme]{Christophe Vuong} \ead{christophe.vuong@telecom-paris.fr}
% \address[ltci]{LTCI, Telecom Paris, Universit\'e Paris-Saclay, 75013, Palaiseau,
%   France} 

\date{}
%
% Abstract.
%
\begin{abstract}
  On any denumerable product of probability spaces, we extend the discrete Malliavin structure for conditionally independent random variables. As a consequence, we obtain the chaos decomposition for functionals of conditionally independent random variables.
 We also show how to derive some concentration results in that framework.
 The Malliavin-Stein method yields Berry-Esseen bounds for U-Statistics of such random variables. It leads to quantitative statements of conditional limit theorems: Lyapunov's central limit theorem, De Jong's limit theorem for multilinear forms. The latter is related to the fourth moment phenomenon. 
  The final application consists of obtaining the rates of normal approximation for subhypergraph counts in random exchangeable hypergraphs including the Erdös-Rényi hypergraph model. The estimator of subhypergraph counts is an example of homogeneous sums for which we derive a new decomposition that extends the Hoeffding decomposition.
  % We also obtain new first-order Poincaré inequality and McDiarmid's inequality.
\end{abstract}

\maketitle{}
\tableofcontents
\section{Introduction}
\label{sec:introduction}
Malliavin calculus is also known as the stochastic calculus of variations. At
the very core of it, it considers a gradient on a measured space. The link
between these the differential geometry and the measure is made through the
so-called integration by parts formula. When the measured space is the Wiener
space, i.e. the set of continuous functions with the Brownian measure, the
gradient generalizes the usual gradient on $\R^{N}$ and the integration by parts
yields an extension of the Itô integral. When the measured space is the set of
point processes on the real half-line, equipped with the law of a Poisson
process, the gradient becomes a difference operator and the integration by parts
is nothing but an avatar of the Mecke formula. It is only very recently that,
concomitantly, the situation where the measured space is a product space, i.e.
if we deal with independent random variables, has been addressed (see
\cite{dung2018poisson,decreusefond:hal-01565240,duerinckx2021size}). By order of
complexity, the next situation which can be analyzed is that of conditionally
independent random variables. This is a very common structure as de Finetti's
theorem says that an infinite sequence of random variables is exchangeable if
and only if these random variables are conditionally independent. This is the
key theorem to develop a theory on random hypergraphs as in
\cite{austin_exchangeable_2008}.

The first definitions of gradient (denoted by $D$) and divergence we introduce
below for conditionally independent random variables, bear strong formal
similarities with those of \cite{decreusefond:hal-01565240}. The difference lies
into the computations which rely heavily on conditional distributions given the
latent variable, which is here called~$Z$. We can then follow the classical
development of the Malliavin calculus apparatus: gradient, divergence, chaos,
number operator and Ornstein-Uhlenbeck semi-group (denoted by $P_{t}$). We can
even describe the dynamics of the Markov process whose infinitesimal generator
is the number operator. At a formal level, the computations are almost identical
to those of \cite{decreusefond:hal-01565240} with expectations replaced by
expectations given~$Z$.

Nevertheless, for more advanced applications, namely functional identities like
the covariance representation formula, we need to introduce a difference
operator (see Definition~\ref{def_discrete-mall-core:1}) which appears more
often than the gradient itself. It is in some sense a finer tool that the
original gradient which is useful to define the Dirichlet structure (the Glauber
process, the infinitesimal generator denoted by $\mathsf L$, etc.) but no more. This is  due
to the fact that $D_{a}D_{a}=D_{a}$, which entails that $\mathsf L$ commutes with $D$, and thus we have $DP_{t}=P_{t}D$ in place of the usual formula
$DP_{t}=e^{-t}P_{t}D$ which is the core formula to derive all functional
inequalities in the Gaussian and Poisson cases. The difference operator $\Delta$
allows to recover the crucial $e^{-t}$ factor (see
Proposition~\ref{prop-discrete-mall:interchange-diff-op-semigroup}).

The prevailing application of Malliavin calculus is nowadays,  the
evaluations of convergence rates via the Stein's method
(\cite{NourdinNormalApproximationsMalliavin2012,Decreusefond2022} and references
therein). The question is to assess a bound of the distance between a target distribution
(more often the Gaussian distribution) and the law of a deterministic transformation of a probability
measure, called the initial distribution.
% The seminal paper \cite{nourdin2009stein} introduced to a new way of proving
% quantitative limit theorems by combining Malliavin calculus and Stein's method,
% which has been fruitful - see the constantly updated resource
% \begin{center}
%   \url{https://sites.google.com/site/malliavinstein/home}
% \end{center}
% for an overview of Malliavin-Stein research line. Stein's method is a collection
% of techniques for assessing probabilistic approximations, the well-known
% application being the central limit theorem \cite{stein1972bound}.

% The Stein's method dates back to the seventies when it was created by C. Stein
% for the convergence towards the one dimensional Gaussian standard
% distribution. It was quickly extended to the convergence towards the Poisson
% distribution (see \cite{barbour2005introduction} and references therein for a
% more complete history).
%
The Dirichlet structure is useful to construct the characterization of the
target distribution and to obtain the so-called Malliavin-Stein representation
formula \cite{DecreusefondSteinDirichletMalliavinmethod2015}. The Malliavin
gradient or the carré du champ operator on the space on which lives the initial
distribution are of paramount importance to make the computations which yield
the distance. In the historical version of the Stein's method, this step was
achieved via exchangeable pairs or biased coupling. One of the key difference
between the Gaussian case and so-called discrete situations (Poisson,
Rademacher, independent random variables) is the chain rule formula: it is only in the
former framework that $D\psi(F)=\psi'(F)\,DF$. For the other contexts, we need to
resort to an approximate chain rule \cite{reinert2010stein}. This is the role
here of Lemma~\ref{lem-discrete-mall:approx-diffusion} and
Lemma~\ref{lem-discrete-mall:approx-diffusion-bis}.
Motivated by the applications to random graphs statistics, we focus here on
normal approximations of $U$-statistics as in
 \cite{barbour1989central,krokowski2017discrete, privault2020normal,rollin2022kolmogorov}. In passing,
we extend the notion of $U$-statistics by allowing the coefficients to depend on the
latent variable instead of being only deterministic. Following the strategy of
\cite{azmoodeh2014fourth}, we establish a fourth moment theorem with remainder for such
functionals. As an application,  we apply our theorems to deduce results of
asymptotic normality of subhypergraph counts in random hypergraphs.

The rest of the paper is organized as follows. The
section~\ref{sec:discrete-mall-cond-indep} lays the foundations of the Malliavin
framework.
% which is an extension of the discrete Malliavin calculus for independent
% random variables \cite{decreusefond:hal-01565240}.
We derive some functional identities in section~\ref{sec:funct-ident},
specifically conditional versions of Poincaré inequality and McDiarmid's
inequality.
% It uses the Malliavin structure instead of a logarithm Sobolev inequality in
% discrete settings.
The section~\ref{sec:normal-approximation} presents results of normal
approximation. In particular, the subsection \ref{sec:partial-fourth-moment}
states a partial fourth moment theorem for U-statistics under mild assumptions.
The aforementioned applications to hypergraph statistics are in
Section~\ref{sec:appli}.
% The proofs are given in the last section.

% \section{}
% \label{sec:mall-calc}

\section{Discrete Malliavin-Dirichlet structure}
\label{sec:discrete-mall-cond-indep}

\subsection{Preliminaries}

Let $A$ be an at most denumerable set equipped with the counting measure, and
define:
\begin{equation*}
  \ell^{2}(A):=\left\{u\, :\,A\to \R,\ \sum_{a\in A}|u_{a}|^{2}<\infty \right\} \text{ and } \langle {u,v} \rangle_{\ell^{2}(A)}:=\sum_{a\in A}u_{a}v_{a}.
\end{equation*}
Let $(\Omega, \mathcal{T}, \mathbb{P})$ be a probability space, $E_{0}$ be a
Polish space and $((E_a, \Upsilon_a),\, a \in A)$ be a family of Polish spaces
such that
\begin{equation}
  \label{eq_discrete-mall-core:1}
  \begin{aligned}
    E_A &= \prod_{a \in A} E_a \\
    \Omega &= E_{0}\times E_{A}.
  \end{aligned}
\end{equation}
The product probability space $E_A$ is endowed with its Borel $\sigma$-algebra
denoted $\Upsilon \subset \mathcal{T}$. Let $Z$ an $E_0$-valued random variable.
By Theorem 10.2.2 \cite{dudley2002real}, all the subsequent conditional
distributions in the paper admit regular versions.
% , and the probability $\mathbf{P}$.
For any subset $B$ of $A$, we denote the set $E_B := \prod_{b \in B} E_b$ and
for $x \in E_A$, $x_B := (x_a, a\in B) \in E_B$ so that for $a \in B$,
$x_a \in E_a$. We denote $x^{B} = (x_{a}, a \in A \sminus B)$. Let
$X := (X_a)_{a \in A}$ be a sequence defined on
$(\Omega, \mathcal{T}, \mathbb{P})$ of conditionally independent random
variables given
% sequence of mutually conditionally independent random variables
$Z$ such that for all $a \in A$, $X_a$ is an $E_a$-valued random variable, i.e.:
$$X_a \operatornamewithlimits{\bigCI}_{Z} (X_b, b \in A \sminus \{a \}),$$
or, equivalently:
\begin{equation*}
  \mathbb{P}(X_a \in \cdot \, | \, \sigma((X_b, b \ne a), Z)) = \mathbb{P}(X_a \in \cdot  \, | \, \sigma(Z)).
\end{equation*}
% In the remainder, it is denoted $\mathbb{P}(\cdot | Z)$.
We denote by $\mathbf{P}$ the law of $X$ and $\mathbf{P}^Z$ the law
$\mathcal{L}(X|Z)$.
%
% Review conditional independence
See chapter 5 of \cite{kallenberg1997foundations} for a thorough review of
conditional independence, and \cite{rao2009conditional} for some limit theorems
for conditionally independent random variables.
We use the notation $\E$ for the expectation of a random variable.
By the disintegration theorem, for $a \in A$, the conditional probability
distribution of $X_a$ given $\sigma(X^{\{a\}}) \vee \sigma(Z)$ admits a regular
version $\mathbf{P}_a$.
% Every Polish space is a Radon space
%%% Very basic examples
% \begin{example}
%   Let $(U_i)_{i \in \N}$ independent uniform random variables, and
%   $X_i = \ind{U_i \le Z}$, with $Z$ an arbitrary random variables lying in
%   $[0, 1]$, then $(X_i)_{i \in \N}$ forms a sequence of conditionally
%   independent random variables given $Z$.
% %
%   The law of $X_i$ given $X^{\{i\}} = X_{\N \sminus \{i\}}$ and $Z$ is a
%   Bernoulli law of parameter $Z$.
% \end{example}
%
% \begin{example}[Copula Gaussian]
%
% \end{example}
%
% %%%
% \todo[]{Example copula Gaussian concentration}
%
For $p \ge 1$, let us denote $L^p(E_A \to \R, \mathbf{P})$ the set of
$p$-th-integrable functions on $E_A$ with respect to the measure $\mathbf{P}$.
It is equipped with the norm $\|\cdot \|_{L^p(E_A \to \R, \mathbf{P})}$ defined
for $f$ a measurable function on $E_A$ by
$\|f \|_{L^p(E_A \to \R, \mathbf{P})} := \int |f(x)|^p \mathbf{P}(\dif x)$. For
the sake of notations, $L^p(E_A)$ stands for the space of $p$-integrable
functionals
\begin{equation*}
  L^p(E_A) := \Bigl\{\omega \mapsto F(X(\omega)) : \omega \in \Omega, F \in L^p(E_A \to \R, \mathbf{P})\Bigr\}.
\end{equation*}
% Then, we define the set of square-integrable functions on $E_A$
% \begin{equation*}
%   L^2(E_A) := \{F: F \longmapsto F(X) \,|\, F \in L^2(E_A \to \R, \mathbf{P}) \}
% \end{equation*}
%% Abuse of notation, omit $X$
In this respect, $L^\infty(E_A)$ is the space of bounded functionals. We shall
write $F$ in place of $F(X)$ for the sake of conciseness.
%
% They may only depend on a finite set of coordinates. Such functionals are
% called cylindrical functionals.
We closely follow the usual construction of Malliavin calculus on that space.
\begin{definition}
  A functional $F$ is said to be cylindrical if there exists a finite subset
  $I\subset A$ and a functional $F_I$ in $L^2(E_I)$ such that
  $\E [|F_I|^2] < +\infty$ and
  \begin{math}%\label{cyl}
    F=F_I\circ r_I,
  \end{math}
  where $r_I$ is the restriction operator:
  \begin{align*}
    r_I\, :\,
    E_A&\longrightarrow E_I\\
    (x_a,a\in A) &\longmapsto (x_a,a\in I).
  \end{align*}
  % This means that $F$ is function of a finite sequence of random variables
  % $(X_a,\, a\in B)$.
%
\end{definition}
It is clear that the set of those functionals $\cyl$ is dense in $L^2(E_A)$.
%
% We give an example that is an extension of the definition of homogeneous sum
% where we replace independent random variables with conditionally independent
% random variables given $Z$.\lau{C'est un exemple de quoi ? Il me semble que le
% point important c'est que l'on peut prendre des $a_{I}$ qui sont $Z$
% mesurables mais ça n'apparaît qu'après dans la remarque. On ne voit pas où ça
% nous mène.} \chris{Je l'introduis pour la suite. En fait c'est un exemple de
% fonctionnelles cylindriques. Je ne savais pas où l'introduire. Comme la
% définition n'est pas bien méchante, j'ai voulu la mettre au début.}
%
We set $L^2 (A \times E_A)$ the Hilbert space of processes which are
square-integrable with respect to the measure
$\sum_{a \in A} \delta_a \otimes \mathbf{P}$, i.e.
\begin{equation*}
  L^2(A  \times E_A) = \{U : \:  \sum_{a \in A} \E  \left[U_a(X)^2 \right] < +\infty\},
\end{equation*}
equipped with the norm and inner product:
\begin{equation*}
  \|U\|_{L^2 (A \times E_A)} := \sum_{a \in A}  \E\left[U_a^{2} \right]  \text{ and } \langle U, V \rangle_{L^2 (A \times E_A)} := \sum_{a \in A} \E \left[U_a V_a \right].
\end{equation*}

%%% Set of simple processes as to use it for cylindrical functionals
\begin{definition}
  The set of simple processes, denoted $\cyl_0(l^2(A))$ is the set of random
  variables defined on $A \times E_a$ of the form
  \begin{equation*}
    U = \sum_{a \in A} U_a \indicator_a,
  \end{equation*}
  for $U_a \in \cyl$.
\end{definition}

\subsection{Malliavin operators}
\label{sec:mall-operators}

% We first define the discrete gradient of cylindrical functionals, and then
% extend the domain of that operator in $L^2(E_A)$ to a larger class of
% functionals.

\begin{definition}[Discrete gradient]
  For $F\in \cyl$, $DF$ is the simple process of $L^2(A\times E_A)$ defined for
  all $a\in A$ by:
  \begin{equation*}
    D_a F := F - \excond{F}{X^{\{a\}}, Z}.
  \end{equation*}
  In particular, $\cyl \subset \dom D$.
  Define the $\sigma$-field $\sigma(X^{\{a\}}) \vee \sigma(Z)$ by
  $\mathcal{G}^a$, so that
  \begin{equation}
    D_a F = F - \excond{F}{\mathcal{G}^a}.
  \end{equation}
\end{definition}
Recall that for $K \subset A$, $X_K = (X_a, a \in K)$ and
$X^K = (X_{a}, a \in A \sminus K)$. We shall write
$\mathcal{G}^K = \sigma(X^K) \vee \sigma(Z)$ and
$\mathcal{G}_K = \sigma(X_K) \vee \sigma(Z)$ for $K$ a subset of $A$.

\begin{lemma}
  \label{lem-discrete-mall:gradient-commut}
  Let $(a, b) \in A^2, \: a \ne b$, for $F \in \dom D$,
  \begin{enumerate}
    \item $D_a D_a F = D_a F$;
    \item $D_{a} D_{b} F = D_{b} D_{a} F$;
    \item $D_{a}\excond{F}{\mathcal{G}^b} = D_{b}\excond{F}{\mathcal{G}^a}$.
  \end{enumerate}
\end{lemma}

\begin{proof}[Proof of lemma~\ref{lem-discrete-mall:gradient-commut}]
  For $(a, b) \in A^2$, with $b \ne a$,
  \begin{align*}
    D_aD_b F &=D_b F - \excond{D_b F}{\mathcal{G}^a}\\
             &= F - \excond{F}{\mathcal{G}^b} - \excond{F}{\mathcal{G}^a} + \excond{\excond{F}{\mathcal{G}^b}}{\mathcal{G}^a} \\
    D_bD_a F &=D_a F - \excond{D_a F}{\mathcal{G}^b} + \excond{\excond{F}{\mathcal{G}^a}}{\mathcal{G}^b}\\
             &= F - \excond{F}{\mathcal{G}^a}  - \excond{F}{\mathcal{G}^b} + \excond{\excond{F}{\mathcal{G}^a}}{\mathcal{G}^b}.\\
  \end{align*}
  We note that:
  \begin{multline*}
    \excond{\excond{F(X)}{\mathcal{G}^a}}{\mathcal{G}^b} \\
    \begin{aligned}
      &=\int \int F(X_{A \sminus \{a, b\}}, x_a, x_b) \mathbf{P}_a ((X_{A \sminus \{a, b\}}, Z), x_b, \dif x_a) \mathbf{P}_b((X_{A \sminus \{a, b\}}, Z), \dif x_b) \\
      &= \int \int F(X_{A \sminus \{a, b\}}, x_a, x_b) \mathbb{P}^{X_b|Z}(Z, dx_b)\mathbb{P}^{X_a|Z} (Z, \dif x_a) \\
      &= \excond{\excond{F(X)}{\mathcal{G}^b}}{\mathcal{G}^a}. \\
    \end{aligned}
  \end{multline*}
  Hence, the equality follows.
\end{proof}

%%%%%%%%%%%%%%%%%%%%%%%%%%%%%%%%%%%%%%%%%%%%%%%%%%%%%%%%%%%%%%%%%%%%%%%%%%%%%%
% Strategy to deploy by a limiting procedure.
The key to the definition of the Malliavin framework is the so-called
integration by parts.

\begin{theorem}[Integration by parts I]
  \label{thm-discrete-mall:ipp-gradient}
  Let $F \in \mathcal{S}$, for every simple process $U$,
  \begin{equation}
    \label{eq-discrete-mall:ipp-gradient}
    \langle DF , U \rangle_{L^2(E_A \times A)} = \E \left[F \sum_{a \in A} D_a U_a \right].
  \end{equation}
\end{theorem}

\begin{proof}[Proof of theorem \ref{thm-discrete-mall:ipp-gradient}]
  % It is an extension of the theorem 2.4 of \cite{decreusefond:hal-01565240}
  % that is proved in the same way.
  We get:
  \begin{align*}
    \langle DF, U \rangle_{L^2(E_A \times A)} &= \E\left[\sum_{a \in A} D_a F U_a \right]  \\
                                              &=   \E\left[\sum_{a \in A} (F - \excond{F}{\mathcal{G}^a})U_a \right] \\
                                              &=\sum_{a \in A}  \E\left[F (U_a - \excond{U_a}{\mathcal{G}^a}) \right] \\
                                              &= \sum_{a \in A} \E \left[F D_a U_a\right],
  \end{align*}
  by self-adjointness of the conditional expectation.
\end{proof}

\begin{corollary}[Closability of the discrete gradient]
  \label{cor-discrete-mall:closability-gradient}
  The operator $D$ is closable from $L^2 (E_A)$ into $L^2(A \times E_A)$.
\end{corollary}
\begin{proof}[Proof of corollary~\ref{cor-discrete-mall:closability-gradient}]
  The proof is analogous to the proof of closability of the gradient in
  \cite[corollary 2.5]{decreusefond:hal-01565240}
\end{proof}

The domain of $D$ in $L^2(E_A)$ is the closure of cylindrical functionals with
respect to the norm:
$$\|F\|_{1,2} := \sqrt{\|F\|_{L^2(E_A)}^2 + \|D F\|_{A  \times L^2(E_A)}^2}.$$

The following lemma gives a way to define square-integrable functionals in
$\dom{D}$ that are not in $\cyl$.

\begin{lemma}% [Closure of the gradient]
  \label{lem-discrete-mall:closure-gradient}
  If there exists a sequence $(F_n)_{n \in \N}$ of elements of $\dom D$ such
  that
  \begin{enumerate}
    \item the sequence converges to $F$ in $L^2(E_A)$,
    \item \label{enum:sup}
          $\sup_{n \in \N} \|D F_n\|_{L^2(E_A \times A)} < +\infty$,
  \end{enumerate}
  then $F$ belongs to $\dom D$ and $DF = \lim_{n \to +\infty} DF_n$.
\end{lemma}

\begin{proof}[Proof of lemma~\ref{lem-discrete-mall:closure-gradient}]
  Let $(F_n)_{n \in \N}$ a sequence in $L^2(E_A)$ with $\mathbf{P}$-a.s. limit
  $F$, then for $a \in A$,
  \begin{align*}
    \E [|D_a F - D_a F_n|^2] &\le \E [|F - F_n|^2] + \E \left[|\excond{F_n}{\mathcal{G}^a} - \excond{F}{\mathcal{G}^a}|^2\right] \\
                             & \le \E [|F - F_n|^2] + \E \left[\excond{|F - F_n|^2}{\mathcal{G}^a}\right] \text{ by Jensen's inequality}\\
                             &= 2 \E [|F - F_n|^2] \xrightarrow{n \to +\infty} 0.
  \end{align*}
  Let $(A_m)_{m \in \N}$ a family of subsets of $A$ such that
  $\bigcup_{m \ge 0} A_m = A$ and $|A_m| = m$, then for all $m \in \N$,
  $(\sum_{a \in A_m} D_a F_n)_{n \in \N}$ converges in $L^2(E_A)$ to
  $\sum_{a \in A_m} D_a F$. We denote by $D^m$ the operator on
  $L^2(E_A \times A)$ such that for $a \in A_m$, $D^m_a = D_{a}$ and otherwise
  $D_a^m$ is the null operator. For $m \in \N$, $(D^m F_n)_{n \in \N}$ converges
  to $D^m F$ in $L^2(E_A \times A)$. Because of (\ref*{enum:sup}), by the
  uniform boundedness principle, $D F$ is in $L^2(E_A \times A)$, and the result
  follows.
\end{proof}
%
%
%%%%%%%%%%%%% Important remark on the degeneracy %%%%%%%%%%%%%%%%%%%%%%%%%%%%%
% For some degenerate cases, there exists $B \subset A$ such that
% \begin{equation}
%   \label{eq-discrete-mall:dependent-Z}
%   Z = X_B\: \mathbf{P} \text{-a.s.}. \tag{DZ}
% \end{equation}
% Then the discrete gradient corresponds to the one in
% \cite{decreusefond:hal-01565240}, i.e. for all $a \in A$,
% $\mathcal{G}^{a} = \sigma(X^{\{a\}})$. Otherwise, for all $a \in A$,
% $\mathcal{G}^{a} = \sigma(X^{\{a\}}, Z)$.
%
  % \begin{remark}
  %   If \eqref{eq-discrete-mall:dependent-Z} holds, the conditional
  %   independence applies to $X^{B}$, since $X_B$ is constant $\mathsf{P}$-a.s.
  %   given $Z$, and every other random variable is independent of a constant
  %   random variable.
  % \end{remark}
%
  % \begin{remark}[Why a degenerate case should exist?]
  %   It is known that a central limit theorem for exchangeable random variables
  %   exists under some natural conditions. That should transfer to the study of
  %   exchangeable random hypergraphs. The keyword may be exchangeable instead
  %   of conditional independence.
  % \end{remark}
%%%%%%%%%%%%%%%%%%%%%%%%%%%%%%%%%%%%%%%%%%%%%%%%%%%%%%%%%%%%%%%%%%%%%%%%%%%%%%
  % See Lachieze-Rey for notations: the difference operator can be found in
  % several papers that deals with discrete structure as it is.
%
\begin{definition}[Divergence operator]
  % From Nualart course (replace relationship with relation)
  The domain of the divergence operator $\dom{\delta}$ in $L^2(E_A)$ is the set
  of processes $U$ in $L^2(E_A \times A)$ such that there exists $\delta U$
  satisfying the duality relation
  \begin{equation}
    \langle DF, U \rangle_{L^2(E_A \times A)} = \E [F \delta U], \text{ for all } F \in \dom D.
  \end{equation}
  % Let define the domain
  % \begin{align*}
  %       \dom \delta &= \{U \in L^2 (A \times E_A): \\
  %        & \exists\, c > 0 , \forall F \in \dom D, |\langle DF, U \rangle_{L^2(E_A \times A)}| \le c \|F\|_{L^2(E_A)}\}.
  %   \end{align*}
  Moreover, for any process $U$ belonging to $\dom \delta$, $\delta U$ is the
  unique element of $L^2 (E_A)$ characterized by that identity. The integration
  by parts formula entails that for every process $U \in \dom \delta$,
  \begin{equation}
    \delta = \sum_{a \in A} D_a U_a.
  \end{equation}
\end{definition}

\begin{definition}[Ornstein-Uhlenbeck operator]
  The Ornstein-Uhlenbeck operator, denoted by $\mathsf{L}$ is defined on its
  domain
  \begin{equation*}
    \dom{\mathsf{L}} = \left\{F \in L^2(E_A): \: \E \left[\left|\sum_{a \in A} D_a F \right|^2\right] <+\infty \right\} \supseteq \cyl
  \end{equation*}
  by
  \begin{equation}
    \mathsf{L} F := - \delta D F = -\sum_{a \in A} D_a F.
  \end{equation}
\end{definition}

\subsection{Chaos decomposition}
\label{sec:chaos-decomp}
The lemma \ref{lem-discrete-mall:gradient-commut} entails a chaos decomposition
of $L^2(E_A)$ similar to the one in \cite{duerinckx2021size}.

\begin{theorem}[Chaos decomposition]
  For any $F \in L^2(E_A)$,
  \begin{equation}
    F = \excond{F}{Z} + \sum_{n=1}^{+\infty} \pi_n (F),
  \end{equation}
  where $(\pi_n)_{n \in \N}$ is a sequence of orthogonal projectors on
  $L^2 (E_A)$.
\end{theorem}

\begin{proof}
  % Beware, the product symbol means multiple composition between functions.
  One can notice that:
  \begin{equation}
    \E [D_a F (X)| \mathcal{G}^a]  =  D_a (\E [F | \mathcal{G}^a]) F(X) = 0, \text{ for all }a \in A.
  \end{equation}
  Let $(A_m)_{m \in \N}$ a family of finite subsets of $A$ such that $|A_m| = m$
  and $\bigcup_{m \in \N} A_m = A$. Let $m \in \N$,
  $\text{Id}_{L^2(E_{A_m})} = \prod_{a \in A_m} (D_a + \E [ \cdot | \mathcal{G}^a])$.
  Indeed, for all $a \in A_m$,
  $\text{Id}_{\dom{D}} = D_a + \E [ \cdot | \mathcal{G}^a]$. Hence, by
  distributivity and by using lemma \ref{lem-discrete-mall:gradient-commut}, the
  identity also reads off:
  $\text{Id}_{L^2(E_{A_m})} = \sum_{n = 0}^{m} \pi_n^m$, where
  \begin{equation}
    \label{eq-discrete-mall::projector-general}
    \pi_n^m :=  \sum_{J \subset A_m, \: |J| = n}\left(\prod_{b \in J} D_b\right) \left(\prod_{c \in A_m \sminus J} \E [ \cdot | \mathcal{G}^c] \right) \quad \forall n \le m.
  \end{equation}
  Let $n \le m$,
  \begin{equation}
    \begin{aligned}
      \pi_n^m \pi_n^m &= \sum_{\substack{I \subset A_m \\ |I| = n}} \sum_{\substack{J \subset A_m \\ |J| = n}} \left(\prod_{b \in I} D_b\right) \left(\prod_{c \in A_m \sminus I} \E [ \cdot | \mathcal{G}^c] \right) \left(\prod_{d \in J} D_d\right) \left(\prod_{e \in A_m \sminus J} \E [ \cdot | \mathcal{G}^e] \right) \\
      % &= \sum_{I \subset A_m, \: |I| = n} \sum_{J \subset A_m, \: |J| = n} \left( \underbrace{\prod_{b \in I} D_b \prod_{e \in A_m \sminus J} \E [ \cdot | \mathcal{G}^e] }_{= 0 \text{ whenever we compose }D_{a'} \E [ \cdot | \mathcal{G}^{a'}]} \right) \left( \prod_{c \in A_m \sminus I} \E [   \cdot | \mathcal{G}^c ] \prod_{d \in J} D_d \right) \\
                      &= \sum_{I \subset A_m, \: |I| = n} \sum_{J \subset A_m, \: |J| = n} \left( \prod_{b \in I} D_b \prod_{e \in A_m \sminus J} \E [ \cdot | \mathcal{G}^e] \right) \left( \prod_{c \in A_m \sminus I} \E [   \cdot | \mathcal{G}^c ] \prod_{d \in J} D_d \right) \\
                      &= \sum_{\substack{I \subset A_m \\ |I| = n}} \left( \prod_{b \in I} D_b \prod_{e \in A \sminus I} \E [ \cdot | \mathcal{G}^c]  \right) \left( \prod_{c \in A_m \sminus I} \E [   \cdot | \mathcal{G}^c ] \prod_{d \in I} D_d \right) \text{ by lemma~}\mbox{\ref{lem-discrete-mall:gradient-commut}}  \\
                      &=  \sum_{I \subset A_m, \: |I| = n} \left(\prod_{b \in I} \prod_{b  \in I} D_b D_b\right) \left(\prod_{c \in A_m \sminus I} \E [\cdot | \mathcal{G}^c] \E [\cdot | \mathcal{G}^c] \right)  = \pi_n^m.
    \end{aligned}
  \end{equation}
  % using \eqref{lem-discrete-mall:gradient-commut}.
  By convention $\pi_n^m (F) = 0$ for $n > m$. Analogously, for $n' \ne n$,
  $\pi_n^m \pi_{n'}^m = 0$.
  The operator $\pi_n^m$ is continuous on $L^2(E_A)$. Hence,
  $(\pi_n^m)_{m \in \N}$ is a well-defined family of projectors on $L^2(E_A)$.
  Moreover, for all $n \in \N$ and $F \in L^2(E_A)$, we have
  $\sup_{m \in \N} \|\pi_n^m(F)\|_{L^2(E_A)} \le \| F\|_{L^2(E_A)}$. Then, by
  the uniform boundedness principle,
  \begin{equation*}
    \sup_{\substack{m \in \N \\ \|F\|_{L^2(E_A)} }} \|\pi_n^m(F)\|_{L^2(E_A)}  < +\infty.
  \end{equation*}
  The pointwise limits of $(\pi_n^m (F))_{m \in \N}$ for $F \in L^2(E_A)$ define
  a bounded linear operator $\pi_n$ on $L^2(E_A)$ for $n \in \N$. Thus:
  \begin{equation}
    L^2(E_A) = \bigoplus_{n=0}^{+\infty} \text{Im } \pi_n.
  \end{equation}
  Given \eqref{eq-discrete-mall::projector-general}, for a functional
  $F \in \dom{\mathsf{L}}$, we have $\pi_0 (F) = \excond{F}{Z}$.
\end{proof}
\begin{lemma}[Spectral decomposition]
  % The Ornstein-Uhlenbeck operator $\mathsf{L}$ satisfies the following
  % properties:
  Let $F \in L^2(E_A)$ of chaos decomposition
  \begin{equation*}
    F = \excond{F}{Z} + \sum_{n=1}^{+\infty} \pi_n (F).
  \end{equation*}
  \label{lem-discrete-mall:chaos-decomposition}
  \begin{enumerate}
    \item We say that $F$ belongs to $\dom{\mathsf{L}}$ whenever
          \begin{equation*}
            \sum_{n=1}^{+\infty} n^2 \| \pi_n(F) \|_{L^2(E_A)} < +\infty.
          \end{equation*}
    \item The operator has a unit spectral gap, i.e. the spectrum of
          $\mathsf{L}$ coincides with $\N_0$.
          \begin{equation}
            \label{eq-discrete-mall:sum-L-nId}
            L^2(E_A) = \bigoplus_{k=0}^{+\infty} \ker( \mathsf{L} + k \textnormal{Id}).
          \end{equation}
    \item It is invertible from
          $L^2_0(E_A) = \{F \in L^2(E_A), \: \excond{F}{Z} = 0\}$ into itself.
  \end{enumerate}
  % \todo[]{Expression of the weight function?}
\end{lemma}

\begin{proof}[Proof of lemma~\ref{lem-discrete-mall:chaos-decomposition}]
  Let us show that $\pi_n$ is in the domain of $\mathsf{L}$ for all $n \in \N$.
  By summability,
  \begin{equation}
    \label{eq-discrete-mall:projector-domain-L}
    \begin{aligned}
      |\sum_{a \in A} D_a \pi_n|^2
      &= \left|\sum_{a \in A} D_a \sum_{I \subset A, \: |I| = n} \left(\prod_{b \in I} D_b\right) \left(\prod_{c \in A \sminus I} \E [ \cdot | \mathcal{G}^c] \right)\right|^2 \\
      &=  \left|\sum_{a \in A} \indicator_I (a) \sum_{I \subset A, \: |I| = n} \left(\prod_{b \in I} D_b\right) \left(\prod_{c \in A \sminus I} \E [ \cdot | \mathcal{G}^c] \right)\right|^2 \\
      &= n^2 \left| \sum_{I \subset A, \: |I| = n} \left(\prod_{b \in I} D_b\right) \left(\prod_{c \in A \sminus I} \E [ \cdot | \mathcal{G}^c] \right) \right|^2 \text{ since }|I| = n \\
      &= n^2 |\pi_n|^2,
    \end{aligned}
  \end{equation}
  so for $F \in L^2(E_A)$, $\pi_n(F) \in \dom{\mathsf{L}}$.
  % Domain of L
  Hence, because of the orthogonality of $(\text{Im }\pi_n)_{n \in \N}$,
  $F \in \dom{L} \iff \sum_{n=1}^{+\infty} n^2 \| \pi_n(F) \|_{L^2(E_A)} < +\infty$.
  With the same calculations, we get $\mathsf{L} \pi_n = -n \pi_n$. The spectrum
  of $-\mathsf{L}$ coincides with $\N$. Then, we deduce that:
  \begin{equation}
    \mathsf{L } = \sum_{n=0}^{+\infty} -n \pi_n,
  \end{equation}
  and $\text{Im}\, \pi_n \subset \ker(\mathsf{L + n \text{Id}})$. Because of the
  orthogonality of the kernels, we get
  $\text{Im}\, \pi_n =\ker(\mathsf{L + n \text{Id}})$.
%%%%%
  Now let us prove the third item.
  The pseudoinverse $\mathsf{L }^{-1}$ is defined on its domain
  $ \{F \in L^2(E_A) : \: \excond{F}{Z} = 0 \}$ and reads
  $\sum_{n=1}^{+\infty} - \frac{\pi_n}{n}$. Then for
  $F \in \{G \in \dom{\mathsf{L}} : \: \excond{G}{Z} = 0 \}, \mathsf{L}^{-1} (\mathsf{L} F) = F$.
\end{proof}

\begin{corollary}
  \label{cor-discrete-mall:mathfrak}
  For $k > 0$ and $J$ a subset of $A$ of cardinal $k$, let us denote by
  $\mathfrak{C}_k$ the space of functionals
  $\phi = \sum_{J \subset A, |J| = k} \psi_J$ such that:
  \begin{itemize}
    \item for every $J \subset A$, $\psi_J$ is $\mathcal{F}_J$-measurable;
    \item for every $K \subset A$, $\excond{\psi_J}{\mathcal{G}_K} = 0$ unless
          $K \subset J$;
  \end{itemize}
  then $\mathfrak{C}_k = \ker( \mathsf{L} + k \textnormal{Id})$.
  % \item For $F$ a functional of conditionally independent random variables,
  % there exists a unique sequence of kernels
  % $(\tilde f_n :\E^n \to \R)_{n \in \N}$ such that:
  % \begin{equation}
  %   F = \tilde f_0 + \sum_{n=1}^{\infty} I_n (F) = \tilde f_0 + \sum_{n=1}^{\infty} J_n (\tilde f_n),
  % \end{equation}
  % where
  % \begin{itemize}
            %       \item $\tilde f_0 = \excond{F}{Z}$
            %       \item
            %       $\forall n \in \N, I_n (F) = J_n (\tilde f_n) = \sum_{K \in (A, n)} w(K) \tilde f_n(X_K) \in \mathfrak{C}_n,$
            %       with $w$ a weight function.
            %     \end{itemize}
\end{corollary}

\begin{proof}[Proof of corollary~\ref{cor-discrete-mall:mathfrak}]
  From \eqref{eq-discrete-mall::projector-general}, for
  $J = (a_1, \ldots, a_n) \subset A$, the component $\psi_J$ is
  $\mathcal{F}_J$-measurable.
  Let us compute the expression of the iterated gradient for $F$ a
  $\mathcal{F}_J$-measurable function:
  \begin{align*}
    \prod_{a \in J} D_a F
    % &= F - \excond{F}{\mathcal{G}^{a_1}} - \excond{F}{\mathcal{G}^{a_2}} + \excond{F}{\mathcal{G}^{a_1} \cap \mathcal{G}^{a_2}} \\
    % &- \excond{F}{\mathcal{G}^{a_3}} + \excond{F}{\mathcal{G}^{a_1} \cap \mathcal{G}^{a_3}} + \excond{F}{\mathcal{G}^{a_2} \cap \mathcal{G}^{a_3}} \\
    % &- \excond{F}{\mathcal{G}^{a_1} \cap \mathcal{G}^{a_2} \cap \mathcal{G}^{a_3}} \ldots\\
      &= \sum_{k=0}^{|J|} (-1)^k \sum_{\substack{K \subseteq J\\|K| = k}}  \excond{F}{\mathcal{G}^K} \\
      &= \sum_{L \subseteq J} (-1)^{|J| - |L|} \excond{F}{\mathcal{G}_L},
  \end{align*}
  where $\mathcal{G}^K = \sigma(X^K) \vee \sigma(Z)$ and $\mathcal{G}_L = \sigma(X_L) \vee \sigma(Z)$.
            %             as one can notice that for every $K \subseteq J$,
            %             $\mathcal{G}_K = \mathcal{F}_{J \setminus K}$.
  %
            %             This is exactly the Hoeffding decomposition of a
            %             $\mathcal{F}_J$-measurable function.
  %
  In this view, we have the inclusion
  $\ker (\mathsf{L} + n \text{Id}) = \text{Im} \,\pi_n \subset \mathfrak{C}_n$
  for $n \in \N$.

  Conversely, let $\phi$ for which the properties above hold.
  \begin{align*}
    \mathsf{L} \phi &= -\sum_{a \in A} D_a \sum_{J \subset A, |J| = n} \psi_J \\
                    &=  -\sum_{a \in A}\sum_{J \subset A, |J| = n} (\psi_J - \excond{\psi_J}{\mathcal{G}^a}) \\
                    % &=-\sum_{a \in A}\sum_{\substack{J \subset A, |J| = n \\ a \in J}} (\psi_J - \excond{\psi_J}{\mathcal{F}_{A \setminus \{a\} } }) \\
                    &=-\sum_{k \in A}\sum_{\substack{J \subset A, |J| = n \\ a \in J}} \psi_J \text{ because } \excond{\psi_J}{\mathcal{F}_{A \setminus \{a\} } } = 0 \text{ for }J \not\subset A \sminus \{a\} \\
                    &= -n \sum_{J \in A, |J| = n} \psi_J = -n \phi.
  \end{align*}
  Therefore, $\mathfrak{C}_n = \ker (\mathsf{L} + n \text{Id})$ for $n \ge 1$.
\end{proof}

\subsection{Dirichlet structure}

The map $\mathsf{L}$ can be viewed as the generator of a Glauber dynamics where
the index set is a finite set of random variables indexed by $A_m$ for $m > 1$.
% i.e. for $f \in L^2(E_A, \R; \mathbf{P})$ such that the functional
% $F : f \longmapsto f(X)$ belongs to the domain of $\mathsf{L}$ defined earlier
% and $x \in E_A$,
% \begin{equation*}
%   \mathsf{L} f(x) = \sum_{a \in A} \excond{\Delta_a f(x,X'_a)}{X=x, Z}
% \end{equation*}
%
For practical term, we introduce a new index $\partial$ and
$X_\partial = Z\: \mathbb{P}$-a.s..
% Glauber dynamics
%
\begin{definition}[Modified Glauber process]
  Consider $(N(t))_{t \ge 0}$ a Poisson process on the half-line $[0, +\infty)$
  of rate $|A_m| + 1$. Let
  $(X^{\circ A_m }(t))_{t \ge 0} = (X^{\circ A_m}_{a}(t), t \ge 0, a \in A)$ the
  process valued in $E_{A}$ starting with $X^{\circ A_m}(0) = X$ which evolves
  according to the following rule. At jump time $\tau$ of the process,
  \begin{itemize}
    \item Choose randomly an index $a$ in $A_m \sqcup \{\partial\}$ with equal
          probability.
    \item If $a \neq \partial$, replace $X^{\circ A_m}_{a} (\tau)$ with a
          conditionally independent random variable $X^\backprime_a$ distributed according
          to
          $\mathbf{P}_{a}((X_{A \setminus \{a\}}^{\circ A_m}(\tau), Z), \cdot)$,
          otherwise do nothing.
  \end{itemize}
\end{definition}
That Markov process has for infinitesimal generator $\mathsf{L}^{A_m}$:
\begin{equation*}
  \mathsf{L}^{A_m} F = -\sum_{a \in A_m} D_a F.
\end{equation*}
Our aim is to show that the operator $\mathsf{L}$ is an infinitesimal generator,
letting $m \to +\infty$. We recall the Hille-Yosida theorem
\cite{yosida1995functional}.
\begin{proposition}[Hille-Yosida]
  \label{prop-discrete-mall:hille-yosida}
  A linear operator $L$ on $L^2(E_A)$ is the generator of a strongly continuous
  contraction semigroup on $L^2(E_A)$ if and only if
  \begin{enumerate}
    \item $\dom{L}$ is dense in $L^2(E_A)$;
    \item $L$ is dissipative, i.e. for any $\lambda > 0$, $F \in \dom{L}$,
      $$\| \lambda F - L  F\|_{L^2(E_A)} \ge \lambda \|F\|_{L^2(E_A)};$$
    \item $\operatorname{Im}(\lambda Id- L)$ is dense in $L^2 (E_A)$.
  \end{enumerate}
\end{proposition}

\begin{theorem}%[Glauber dynamics]
  \label{thm-discrete-mall:glauber-infinite}
  $\mathsf{L}$ is an infinitesimal generator on $E_A$ of a strongly continuous
  contraction semigroup on $L^2(E_A)$.
  % $X^\circ$ has $\mathsf{L}$ as infinitesimal generator on $E_A$ and:
  % \begin{equation}
  %   \mathsf{L} F(x) := -\sum_{a \in A \cup \partial} \int_{E_a} \left(F(x^{ \{a \}}, x'_a)- F(x)) \mathbf{P}_a (x_{A \cup \partial \sminus \{a\}}, \dif x'_a)\right),\\
  % \end{equation}
  % with $X_{\partial} = Z \: \mathbb{P}^Z$-a.s. and $E_\partial = \latspace$.
\end{theorem}

\begin{proof}[Proof of theorem~\ref{thm-discrete-mall:glauber-infinite}]
  We know that $\mathcal{S}$ is dense in $L^2 (E_A)$. As
  $\dom{\mathsf{L}} \supset \mathcal{S}$, it is also dense in $L^2 (E_A)$. Let
  $A_m$ an increasing sequence (with respect to $\subset$) of subsets of $A$
  such that $\bigcup_{n \ge 1} A_m = A \cup \partial$ and $|A_m| = m$. Then
  $(\mathcal{F}_{A_m})_{m \in \N}$ is a filtration. For $F \in L^2(E_A)$, let
  $F_m = \excond{F}{\mathcal{F}_{A_m} }$. Since $(F_m)_{m \in \N}$ is a
  square-integrable $\mathcal{F}_A$-martingale, $(F_m)_{m \in \N}$ converges
  both almost surely and in $L^2(E_A)$ to $F$. For any $m \in \N$, $F_m$ depends
  only on $X_{A_m}$. Because of the conditional independence of the random
  variables $X_a$ given $X_\partial$, for all $a \in A$, we get that
  $D_a F_m = \excond{D_a F}{\mathcal{F}_{A_m}}$. Using that $\mathsf{L}_{A_m}$
  is dissipative for all $m \in \N$, we have:
  \begin{align*}
    \lambda^2 \|F_m\|^2_{L^2(E_A)} &\le \| \lambda F_m - \mathsf{L}^{A_m} F_m \|^2_{L^2(E_A)} = \E \left[\left(\lambda F_m + \sum_{a \in A_n} D_a F_m\right)^2\right] \\
                                   &= \E \left[\left(\lambda F_m + \sum_{a \in A} D_a F_m\right)^2\right] \text{ because }D_a F_m = 0, \:\text{if } a \notin A_m.\\
                                   &= \E \left[ \E \left[ \lambda F + \sum_{a \in A} D_a F \, \bigg| \, \mathcal{F}_{A_m}  \right]^2\right].\\
  \end{align*}
  It means that the operator $\mathsf{L}$ is dissipative. Thus, by the
  Hille-Yosida theorem, $\mathsf{L}$ is the infinitesimal generator of a
  strongly continuous contraction semigroup on $L^2(E_A)$ denoted $P$.
\end{proof}

% \todo[]{Lemma convergence semigroups?}
\begin{lemma}
  \label{lem-discrete-mall:convergence-semigroup-processes}
  Let $F \in L^2(E_A)$, then
  \begin{equation*}
    \excond{F(X^{\circ A_m})}{X,Z} = P^{A_m}_t F \xrightarrow{\mathbf{P}-a.s.} P_t F
  \end{equation*}
  and
  \begin{equation*}
    X^{\circ A_m} \xrightarrow{d} X^\circ.
  \end{equation*}
\end{lemma}

\begin{proof}[Proof of lemma~\ref{lem-discrete-mall:convergence-semigroup-processes}]
  The theorem 17.25 of \cite[Trotter, Sova, Kurtz, Mackevi\v
  {c}ius]{kallenberg1997foundations} gives the convergence in distribution of
  $X^{\circ A_m}$ towards $X^\circ$ the Markov process associated to
  $\mathsf{L}$, and the almost sure convergence of the semigroup.
\end{proof}

These are pieces of the Dirichlet structure with invariant measure $\mathbf{P}$
that we complete with the carré du champ operator.
%%% Domain stable by product
Here, we note that $\cyl$ is an algebra which is a core of $\dom{\mathsf{L}}$.
% Assume that we are given a vector subspace $\mathcal{A}$ of the domain
% $\dom{\mathsf{L}}$ such that for every pair $(F, G)$ of functionals of
% $\dom{\mathsf{L}}$ in $\mathcal{A}$, the product $FG$ is in the domain
% $\dom{\mathsf{L}}$, then $\mathcal{A}$ is an algebra. Given the algebra
% $\cyl$, we extend this core to an algebra $\mathcal{A}$ maximal in the sense
% of inclusion. \lau{We define the operator on that algebra.} \todo[]{To
% reformulate} The next lemma pinpoints one of the connections between the
% Dirichlet form and the infinitesimal generator. For further details, we report
% to \cite{bakry2013analysis}.
\begin{definition}[Carré du champ operator]
  Let $F, G \in \cyl$. The bilinear map
    $$\Gamma(F,G) := \frac{1}{2} \left\{\mathsf{L} (FG) - F \mathsf{L}  G - G \mathsf{L}  F\right\}$$
    is well-defined, and called carré du champ operator of the Markov generator
    $\mathsf{L}$.
  \end{definition}
  By an argument of density, there exists an algebra $\mathcal{A} \supset \cyl$
  maximal in the sense of inclusion such that the carré du champ operator acts
  on it.

\begin{definition}[Dirichlet structure]
  The associated Dirichlet structure defined on $(E_A, \Upsilon, \mathbf{P})$ is
  given by the quadruple $(X^\circ, \mathsf{L}, (P_t)_{t \ge 0}, \mathcal{E})$
  where $X^\circ$ is a Markov process with values in $E_A$ whose infinitesimal
  generator is $\mathsf{L}$ and its semigroup is $P$, i.e. for any
  $F \in L^\infty (E_A)$:
  \begin{equation*}
    % &P_t F = \excond{F(X^\circ (t))}{X^\circ (0) = X} \\
    \frac{\dif }{\dif t} P_t F = (\mathsf{L} P_t) F.
    % &= (P_t \mathsf{L}) F.
  \end{equation*}
  Furthermore, $\mathbf{P}^Z$ is the invariant (or stationary) distribution of
  $X^\circ$ given $Z$ and the Dirichlet form is defined by
  \begin{equation*}
    \mathcal{E} (F, G) = \E [\Gamma (F, G)].
  \end{equation*}
\end{definition}

It comes with the classical properties entailed by the spectral decomposition of
$\mathsf{L}$, including the Mehler's formula.

\begin{lemma}[Mehler's formula]
  \label{lem-discrete-mall:pseudo-ergodicity}
  For any $F \in L^2(E_A)$,
  \begin{enumerate}
    \item \label{P_t_decomposition}
          \begin{equation}
            \label{eq-discrete-mall:P_t-exp-Markov-Proc}
            \begin{aligned}
              P_t F &= \excond{F}{Z} + \sum_{n=1}^\infty e^{-nt} \pi_n (F) \\
                    &= \excond{F(X^\circ(t))}{X},
            \end{aligned}
          \end{equation}
          In particular $P_t F \in \dom{\mathsf{L}} \cap \dom{\mathsf{L}^{-1}}$.
    \item \begin{equation*} \lim_{t\to \infty } P_{t}F(X)= \excond{F(X)}{Z}.
      % \text{ and } (X^\circ(0) | Z) \stackrel{\textnormal{d}}{=}\P^Z
      % \Longrightarrow (X^\circ(t)| Z) \stackrel{\textnormal{d}}{=} \P^Z.
    \end{equation*}
    \item The pseudoinverse of $\mathsf{L}$ can be written:
          \begin{equation*}
            \mathsf{L}^{-1} F := -\int_{0}^{+\infty} P_{t} F \dif t.
          \end{equation*}
  \end{enumerate}
\end{lemma}

\begin{proof}[Proof of lemma~\ref{lem-discrete-mall:pseudo-ergodicity}]
  Since formally $P_t = e^{-t \mathsf{L}}$, we get the first line of
  \eqref{eq-discrete-mall:P_t-exp-Markov-Proc} from the spectral decomposition
  of $\mathsf{L}$. The second line is deduced from the definition of the Glauber
  dynamics and by passing to the limit.
            %             By invariance of $\mathbf{P}^Z$ for $P$,
            %             \todo[]{Ethier-Kurtz}
  Then,
  \begin{align*}
    \excond{F}{Z} - F &= \lim_{t \to +\infty} P_t F - P_0 F \\ &= \int_{0}^{+\infty} \frac{\dif }{\dif t} P_t F \dif t \\ &= \mathsf{L} \left(\int_{0}^{+\infty} P_t F \dif t\right).
  \end{align*}
  Taking $\excond{F}{Z} = 0$, we get the expression of the pseudoinverse.
\end{proof}

\begin{remark}
  By the chaos expansion, $P_t F$ can be defined as the limit in $L^2(E_A)$ of
  elements $(P_t F_n)_{n \in \N}$ for $F_n$ in $\mathcal{S}$.
  Hence, it is sufficient to define the semigroup acting on a functional of some
  finite vector of random variables $X_B$, using the definition of the Glauber
  dynamics entailed by it.
\end{remark}

The infinitesimal generator satisfies another integration by parts formula due
to the Dirichlet structure which is the key to investigating the so-called
fourth moment phenomenon.
% \begin{definition}
%   A Dirichlet structure defined on $(E, \mathcal{G}, \nu_0)$ is given by a
%   quadruple $(X^\circ, \mathsf{L}, (P_t)_{t \ge 0}, \mathcal{E})$ where
%   $X^\circ$ is a strong Feller process with values in $E$ whose infinitesimal
%   generator is $\mathsf{L}$ and its semi-group is $P$, i.e. for $f: E \to \R$
%   sufficiently regular
%   \begin{align*}
%     P_t f (x) &= \excond{f(X^\circ (t))}{X^\circ (0) = x} \\
%     \frac{\dif }{\dif t} P_t f(x) &= \mathsf{L} (P_t f)(x) \\
%     &= P_t (\mathsf{L} f)(x).
%   \end{align*}
%   Furthermore $\nu_0$ is the stationary and invariant distribution of
%   $X^\circ$ and the Dirichlet form is defined by
%   \begin{equation*}
%     \mathcal{E} (f, g) = \frac{\dif}{\dif t} \left. \int_E P_t f(x) g(x) \dif \nu_0 (x) \right|_{t = 0}.
%   \end{equation*}
% \end{definition}

\begin{lemma}[Integration by parts II]
  For $(F, G) \in \mathcal{A}^2$,
  \begin{equation}
    \mathcal{E} (F, G) = - \E [F \mathsf{L} G].
  \end{equation}
\end{lemma}

%
% The Mehler's formula is a central device as to derive functional identities.

We introduce to the difference operator which is associated to the
Malliavin-Dirichlet structure at hand. That difference operator serves the same
purpose as in \cite{lachieze2017new} and \cite{dung2021rates} for computations
in the proofs of the limit theorems. % In the following, we shall write $X'$ the
% copy of $X$ which is conditionally independent of $X$ given $Z$ whose
% coordinates are $X'_a$ for $a \in A$.

%
\begin{definition}[Difference operator]
  \label{def_discrete-mall-core:1}
Let $F\, :\, E_{A}\to \R$, for $a \in A$, we introduce the operator
  \begin{align*}
    \deltaa F \,:\,   E_{A} \times E_a &\longrightarrow \R \\
    (x, x'_a) &\longmapsto f(x) - f(x^{\{a\}}, x'_a).
  \end{align*}
  For the sake of conciseness, we shall write
  $F^{\{a\}'} = F(X^{\{a\}}, X'_a)$.
\end{definition}

\begin{lemma}
  \label{lem_discrete-mall-core:diffOperators}
  For $F$ a functional in $\dom{D}$, the gradient also reads as:
  \begin{equation}
    D_a F = \espxz{\deltaa F(X,X'_{a})},
  \end{equation}
  where $X'_a$ has the law of  $X_a$ given $Z$ and is conditionally independent of $X^{\{a\}}$ given $Z$. Similarly,
  \begin{equation}
    \label{eq-discrete-mall::carré-champ-F-G}
    % \Gamma(F, G) = \frac{1}{2} \sum_{a \in A} \excond{(F(X^{\{a\}}, X'_a)-
    % F(X)) (G(X^{\{a\}}, X'_a) - G(X))}{X, Z}.
    \Gamma(F, G) = \frac{1}{2} \sum_{a \in A} \espxz{\Bigl(\deltaa F(X,X'_{a})\Bigr) \Bigl(\deltaa G(X,X'_{a})\Bigr)}.
  \end{equation}
\end{lemma}
\begin{proof}[Proof of lemma~\ref{lem_discrete-mall-core:diffOperators}]
 We have
  \begin{equation*}
\excond{F}{{\mathcal G}^{a}}=\int F(X^{\{a\}},x_{a})\P_{a}(\dif x_{a}).
\end{equation*}
Since $\sigma(X_{a})$ is independent of $\sigma(X^{\{a\}})$ given $\sigma(Z)$,
we obtain
\begin{equation*}
\excond{F}{{\mathcal G}^{a}}=\int F(X^{\{a\}},x_{a})\mathbb{P}^{X_a | Z}(\dif x_{a}).
\end{equation*}
Eqn.\eqref{eq-discrete-mall::carré-champ-F-G} is proved similarly.
\end{proof}

\section{Functional identities}
\label{sec:funct-ident}

% \todo[]{Shorten} We show that the Malliavin structure yields classical
% functional identities which are of independent interest.
This section is devoted to classical functional identities obtained in the
Malliavin framework. We follow the approach of
\cite{houdre2002concentration} using a covariance identity based on difference
operators to deduce concentration inequalities.
% It relies on the commutation relation of the difference operator and the
% semigroup operator.
\begin{proposition}
  \label{prop-discrete-mall:interchange-diff-op-semigroup}
  For $F \in L^2(E_A)$ and $a \in A$, then:
  \begin{equation}
    \label{eq-discrete-mall:commutation-D-Pt}
    D_a (P_t F) =  e^{-t} \excond{\deltaa F (X^\circ(t), X'_a)}{X, Z}
  \end{equation}
  where $X'$ has the law of $X$ given $Z$.
\end{proposition}

\begin{proof}[Proof of proposition~\ref{prop-discrete-mall:interchange-diff-op-semigroup}]
  We consider the Glauber dynamics with index set a finite subset $A_m$ of $A$,
  as the construction of process $(X^{\circ A_m}(t))_{t \in \R^+}$ is explicit
  in that case. Let $a \in A_m$, we denote by $N_a$ the Poisson process of intensity 1 which
  represents the life duration of the $a$-th component in the dynamics of
  $X^{\circ A_m}(t)$, so:
  \begin{equation*}
    X^{\circ A_m}_a(t) = \ind{\tau_a \ge t} X_a + \ind{\tau_a < t} X^\backprime_a,
  \end{equation*}
  where $\tau_a = \inf \{t \ge 0, \: N_a(t) \ne N_a(0)\}$ is the life duration
  of the $a$-the component of the original sequence, exponentially distributed
  with parameter 1 (independent of everything else) and $X^\backprime_a$ is
  conditionally independent of $X$ given $Z$.
  Then:
  \begin{multline*}
    D_a P_t^{A_m}  F = P_t^{A_m} F - \excond{P_t^{A_m} F}{\g{a}} \\
    \begin{aligned}
      % &= P_t^{A_m} F - \excond{\espxz{F(X^{\circ A_m} (t))\ind{t \le \tau_a} } }{\g{a}} -  \excond{\espxz{F(X^{\circ A_m} (t))\ind{t > \tau_a}}  }{\g{a}} \\
      &= P_t^{A_m} F - \excond{\espxz{F(X^{\circ A_m} (t))} \ind{t \le \tau_a} }{\g{a}} -  \espxz{F(X^{\circ A_m} (t))\ind{t > \tau_a}} \\
      &= \espxz{F(X^{\circ A_m}(t)) \ind{t \le \tau_a}} - \excond{\espxz{F(X^{\circ A_m} (t))} \ind{t \le \tau_a} }{\g{a}} \\
      &= e^{-t} \excond{\deltaa F (X^{\circ A_m}(t), X'_a)}{X, Z}
    \end{aligned}
  \end{multline*}
because the law of $X^\backprime_a$ given $X$ is the same as the one of $X'_a$ given $X$.

  On one hand,
  \begin{equation*}
    D_a P_t^{A_m} F  \xrightarrow{\mathbb{P}-a.s.} D_a P_t F.
  \end{equation*}

  On the other hand, by the Skorohod's representation theorem, there exist
  copies of $X^{\circ A_m}$ and $X^\circ$ on a common probability space
  $(\tilde \Omega, \mathcal{\tilde T}, \mathbb{\tilde P})$ such that the
  sequence $(X^{\circ A_m})_{m \in \N}$ converges to $X^\circ$
  $ \mathbb{\tilde P}$-a.s. As the whole structure is invariant by copy, we can
  suppose the almost sure convergence on $(\Omega, \mathcal{T}, \mathbb{P})$,
  and the relation passes to the limit.
\end{proof}

\begin{remark}
  In the case, we have only one random variable (or one particle), then the commutation relation simplifies to $D_a (P_t F) = D_a$.
\end{remark}

% \begin{remark}
%  The commutation of $D_a$ and $\L$ on $\dom{\L}$ entails the commutation of $D_a$ and $P_t$.
%  We note:
%  \begin{equation*}
%   P_t (D_a F) = P_t \left(\excond{\deltaap F}{X}\right),
%  \end{equation*}
%  which is different from the right-hand side of \eqref{eq-discrete-mall:commutation-D-Pt}.
% \end{remark}
% That commutation relation prepares the proofs of the following concentration
% results. We note that they can also be used for functionals of independent
% random variables by considering $Z$ to be $\mathbb{P}$-a.s. constant.
%
\begin{corollary}[Conditional covariance identity]
  \label{cor:cov-identity}
  For any $F, G \in L^2(E_A)$, then:
            %             \begin{equation}
            %             \cov (F, G) = \int_{0}^{\infty} e^{-t} \sum_{a \in A} \E [(D_a F) (G - G^{\{a\}'}) (X^\circ(t))] \dif t.
            %             \end{equation}
            %             and:
  \begin{equation}
    \cov (F, G | Z) = \int_{0}^{\infty}  e^{-t} \sum_{a \in A} \excond{(D_a F) (\deltaa G(X^\circ(t), X'_a))}{Z} \dif t.
  \end{equation}
\end{corollary}

\begin{proof}[Proof of corollary~\ref{cor:cov-identity}]
  We use the following conditional covariance formula analogous to the
  covariance formula:
  \begin{equation}
    \cov (F,G | Z) = \excond{FG}{Z} = \excond{F \L \L^{-1} G }{Z}.
  \end{equation}
  By the integration by parts I \eqref{eq-discrete-mall:ipp-gradient} which also holds with conditional expectation given $Z$, we get:
  \begin{align*}
    \excond{F \L \L^{-1} G}{Z} &= -\sum_{a \in A}\excond{(D_a F)  (D_a \L^{- 1} G)}{Z} \\
                                              &= -\sum_{a \in A} \excond{(D_a F)  ( D_a \int_{0}^{\infty} P_t G \dif t)}{Z} \\
                                              &= -\sum_{a \in A} \excond{\left(D_a F \right)  \left(\int_{0}^{\infty} D_a P_t G \dif t \right)}{Z} \\
                                              &= -\int_{0}^{\infty}  e^{-t} \sum_{a \in A} \excond{(D_a F) \espxz{\deltaa G(X^\circ(t), X'_a)}}{Z}
                                                \dif t,
  \end{align*}
  using \eqref{eq-discrete-mall:commutation-D-Pt}.
\end{proof}
As an immediate consequence of the spectral gap, we find another proof of the
Efron-Stein inequality which is of independent interest.
\begin{proposition}
  \label{prop-discrete-mall:efron-stein-spectral}
  If $F \in \mathfrak{C}_p$ then
  \begin{equation*}
    \var[F] = \frac{1}{p} \mathcal{E}(F) = \frac{1}{p} \|D F\|_{L^2(E_A)}.
  \end{equation*}
  Moreover, if there exist $F_1, \ldots, F_m \in L^2(E_A)$ such that
  $F = \sum_{p=1}^m F_p$ with $F_p \in \mathfrak{C}_p$ for
  $p \in \llbracket 1, m \rrbracket$, then:
  \begin{equation}
    \label{eq-discrete-mall:efron-stein-finite-chaotic}
    \var[F] \le \|D F\|_{L^2(E_A)}.
  \end{equation}
\end{proposition}

\begin{proof}[Proof of proposition~\ref{prop-discrete-mall:efron-stein-spectral}]
  Let us use the previous covariance identity, we have:
  \begin{align*}
    \var [F] &= \cov (F, F) = \E [\Gamma(F, -\mathsf{L}^{-1} F)] \\
             &= \E \left[\Gamma \left(\sum_{p=1}^m F_p, \sum_{q=1}^m \frac{1}{q} F_q \right) \right] \\
             &= \sum_{p=1}^m \sum_{q=1}^m \frac{1}{q}  \E \left[\Gamma \left( F_p, F_q \right) \right] \\
             &= \sum_{p=1}^m \frac{1}{p} \E \left[\Gamma \left( F_p, F_p \right) \right] \text{ because }\E [\Gamma(F_p, F_q)] = 0 \text{ for }q \ne p.
  \end{align*}
  It yields the inequality \eqref{eq-discrete-mall:efron-stein-finite-chaotic}
  noting that $\Gamma(F_p, F_p) \ge 0$ for all $p > 0$.
\end{proof}
We now deduce the conditional first-order Poincaré inequality for functionals of
conditionally independent random variables. The equivalent for functionals of
independent random variables is rather known as the Efron-Stein inequality in
the literature \cite{efron1981jackknife}.
% \todo[]{correct notations $\Delta$}
\begin{theorem}[Conditional Efron-Stein inequality]
  \label{thm-discrete-mall:efron-stein}
  For $F \in L^2(E_A)$ such that $\excond{F}{Z} = 0$,
  \begin{equation}
    \var [F | Z] \le \excond{\Gamma(F, F)}{Z}.
  \end{equation}
\end{theorem}

\begin{proof}[Proof of theorem~\ref{thm-discrete-mall:efron-stein}]
  The conditional covariance formula yields
  \begin{align*}
    \var [F | Z] &= \int_0^{\infty} e^{-u} \sum_{a \in A} \excond{(D_a F )(\deltaa F)(X_u^\circ, X'_a)}{Z} \dif u \\
    % &= \int_0^{\infty} e^{-u} \sum_{a \in A} \excond{\excond{\deltaa F(X, X'_a)}{X, Z}(\deltaa F)(X_u^\circ, X'_a)}{Z} \dif u \\
                 &\le  \int_0^{\infty} e^{-u}  \sqrt{\sum_{a \in A}\E [(D_a F )^2 | Z]} \sqrt{\sum_{a \in A}\E[\espxz{(\deltaa F)(X_u^\circ, X'_a)}^2 | Z]} \dif u. \\
  \end{align*}
  The invariance of $\mathbf{P}^Z$ under the Glauber dynamics entails that
  \begin{equation*}
    \sum_{a \in A}\E \left[\espxz{(\deltaa F)(X_u^\circ, X'_a)}^2 | Z \right] = \sum_{a \in A}\E[(D_a F )^2 | Z].
  \end{equation*}
  Hence,
  \begin{equation*}
    \var[F | Z] \le \E [\Gamma(F, F) | Z],
  \end{equation*}
  proving the theorem.
\end{proof}

% Using the same framework, \cite{huang2021nonlinear} proves matrix versions of
% several related concentration inequalities. The article considers an invariant
% measure that is a product measure for which a Bakry–Émery criterion holds. If
% we consider the product measure to be the one associated to
% $\mathcal{L}(X | Z)$, we get the extended result for matrix functionals.
% \todo[]{Law of $Z$?}
% \begin{remark}
%   That means that for square-integrable function,
%   \begin{equation*}
%     \var[F - \excond{F}{Z}] \le \|D F\|_{L^2(E_A)}.
%   \end{equation*}
%   It can be considered as conditional first-order Poincaré inequality. We can
%   actually derive some further inequalities because of the variance
%   decomposition with projection:
%   \begin{equation*}
%     \var [F] = \var [F - \excond{F}{Z}] + \var [\excond{F}{Z}].
%   \end{equation*}
%   Indeed, if $F$ is an homogeneous sum then, knowing $Z$, $\excond{F}{Z}$ is a
%   functional of independent random variables, so we also get a Poincaré
%   inequality, then provided that $\E [F] = 0$:
%   \begin{equation*}
%     \var[F] \le \|D F\|_{L^2(E_A)} +  \|\tilde D (\excond{F}{Z})\|_{L^2(\tilde E_A)}.
%   \end{equation*}
%   where $\tilde D$ is the discrete gradient for independent random variables
%   in \cite{decreusefond:hal-01565240}.
% \end{remark}
%
We find a version of the McDiarmid's inequality for conditionally independent
random variables.
\begin{theorem}[Conditional McDiarmid's inequality]
  \label{thm-discrete-mall:conditional-mcdiarmid}
  Let $F$ be a square-integrable functional such that for all $a \in A$:
  $$\sup_{\substack{x^{\{a\}} \in E_{A \sminus \{a\}} \\ x'_a \in E_a}} |F(x^{\{a\}}, x'_a) - F(x)| \le d_a.$$
  For any $x > 0$, we have the inequality:
  \begin{equation}
    \mathbb{P}(F(X) - \excond{F(X)}{Z}\geq x | Z) \le \exp \left(-\frac{x^2}{2\sum_{a \in A} d_a^2}\right).
  \end{equation}
\end{theorem}

Our strategy of proof is different from the original McDiarmid's original proof
in \cite{mcdiarmid1989method}.

\begin{proof}[Proof of theorem~\ref{thm-discrete-mall:conditional-mcdiarmid}]
  We assume that $F = F(X)$ is a bounded random variable verifying
            %             $\E[F] = 0$.
  $\espz{F} = 0$. Using the inequality:
  \begin{equation}
    |e^{tx} - e^{ty}| \le \frac{t}{2} |x - y| (e^{tx} + e^{ty}) \quad \forall \: x, y \in \R.
  \end{equation}
  We have:
  \begin{align*}
    |\deltaa e^{t F}(X,X'_a)| &= |e^{tF} - e^{t F^{\{a\}'}}| \\
                              &\le \frac{t}{2} |\deltaa F(X, X'_a)|\left(e^{tF} + e^{t F^{\{a\}'}}\right). \\
  \end{align*}
  Applying the covariance identity, it yields:
  \begin{align*}
    \E [F e^{t F} | Z]
    &= \int_0^{\infty} e^{-u} \sum_{a \in A} \E [D_a e^{t F}\deltaa F (X_u^\circ, X'_a) | Z] \dif u \\
    &\le \int_0^{\infty} e^{-u} \sum_{a \in A} \E \left[\espxz{|\deltaa e^{t F} (X, X'_a)|} \deltaa F (X_u^\circ, X'_a) | Z \right] \dif u \\
    &\le \frac{t}{2} \int_0^\infty e^{-u}  \sum_{a \in A} \E \left[|\deltaa  F(X, X'_a) |e^{tF} |\deltaa F (X_u^\circ, X'_a) | Z\right] \dif u\\
    &+ \frac{t}{2} \int_0^\infty e^{-u} \sum_{a \in A} \excond{|\deltaa  F(X, X'_a)| e^{t F^{\{a\}'}} |\deltaa(X_u^\circ, X'_a) |}{Z} \dif u\\
    % &\le  \frac{t}{2} \sqrt{\sum_{a \in A} \E \left[|\deltaa F (X, X'_a) |^2 e^{tF}|Z \right]} \sqrt{ \int_{0}^\infty e^{-u} \sum_{a \in A} \E \left[ | \deltaa(X_u^\circ, X'_a)|^2 e^{tF} | Z\right] \dif u}  \\
    % &+\frac{t}{2} \int_{0}^\infty e^{-u} \sqrt{\sum_{a \in A} \E \left[|\deltaa F (X, X'_a)  |^2 e^{t F^{\{a\}'}} | Z\right]  \sum_{a \in A} \E \left[|\deltaa(X^{\circ}_u, X'_a) |^2 e^{t F^{\{a\}'}} | Z\right]}\dif u,\\
  \end{align*}
  by using the Jensen's inequality for conditional expectation in the second
  inequality.
            %             \begin{align*}
            %               \E [F e^{t F}] &\le  \frac{t}{4} \sqrt{\sum_{a \in A} \E \left[|F - F^{\{a\}'}|^2 e^{tF}\right]} \sqrt{ \int_{0}^\infty e^{-u} \sum_{a \in A} \E \left[ |F - F^{\{a\}'}|(X_u^\circ)^2 e^{tF}\right] du}  \\
            %               &+\frac{t}{2}  \sqrt{\sum_{a \in A} \E \left[|F - F^{\{a\}'} |^2 e^{2t F^{\{a\}'}/2}\right]} \sqrt{\int_{0}^\infty e^{-u} \sum_{a \in A} \E \left[|F - F^{\{a\}'}|^2(X^{\circ}_u) e^{2t F^{\{a\}'}/2}\right]du}.\\
            %             \end{align*}
  Since $|\deltaa F(X, X'_a)|^2 \le d_a$, $|\deltaa F(X^\circ_u, X'_a)| \le d_a$ for all $u \in \R^+$ and
  $\E [e^{t F^{\{a\}'}} | Z] = \E [e^{t F} | Z]$, this shows that:
  \begin{equation*}
    \E [F e^{t F} | Z] \le \left(\sum_{a \in A} d_a^2 \right) t \E[e^{t F} | Z] = t K^2  \E [e^{t F} | Z],
  \end{equation*}
  where $K^2 := \sum_{a \in A} d_a^2$. Thus, in all generality for $F$ bounded:
  \begin{align*}
    \log \E [e^{t(F- \E[F])} | Z] &= \int_{0}^t \frac{\E [(F - \E[F | Z]) e^{s(F - \E[F |Z])} | Z]}{\E[e^{s(F - \E[F])} ]} \dif s\\
                                  &\le K^2 \int_{0}^t s \dif s = \frac{t^2}{2} K^2,
  \end{align*}
  hence:
  \begin{align*}
    e^{tx} \mathbb{P}(F - \E[F | Z] > x | Z) &\le \E [e^{t(F - \E[F | Z])} | Z] \\
                                             &= e^{t^2 K^2 / 2}, \quad t \geq 0,\\
  \end{align*}
  and:
  \begin{equation*}
    \mathbb{P}(F - \E[F | Z] \geq x | Z) \le e^{\frac{t^2}{2} K^2 - tx}, \quad t \ge 0.
  \end{equation*}
  The minimum of the right-hand side is obtained for $t = x / K^2$. If $F$ is
  not bounded, the conclusion holds for $F_n = \max (-n, \min(F, n))$,
  $n \ge 0$, and $(F_n)_{n \in \N}$ converges
  $\mathbb{P}$-a.s. to $F$. Hence:
  \begin{equation*}
    \mathbb{P}(F - \E[F | Z] \geq x | Z) \le \exp\left(-\frac{x^2}{2 K^2} \right) = \exp \left(-\frac{x^2}{2\sum_{a \in A} d_a^2}\right).
  \end{equation*}
  The proof is thus complete.
\end{proof}

% \todo[]{Eliminate remark}
% \begin{remark}
%   That is another way to prove McDiarmid's inequality. In the case, $A = \N$
%   it can be proved by considering martingales adapted to the filtration
%   $(\mathcal{F}_n)_{n \in \N}$ with
%   $\mathcal{F}_n = \sigma(X_{0}, \ldots X_n, Z)$ for $n \in \N$. Our theorem
%   extends the inequality to the case where there is no order on $A$.
%   %   No need of a timeline
% \end{remark}

\section{Applications to normal approximation}
\label{sec:normal-approximation}

% We recall in a lemma the Stein-Dirichlet representation formula for the normal
% distribution, also known as semigroup method or smart-path formula
% \cite{DecreusefondSteinDirichletMalliavinmethod2015}. We denote by $\Normal(0,1)$ the standard normal
% distribution.
%
The goal is to bound for instance the 1-Wasserstein distance
\begin{equation*}
  d_W (\mathcal{L}(F(X)), \mathcal{L}(Y))|:= \sup_{h \in \mathcal{H}} |\E[h(F(X))] - \E [h(Y)]|
\end{equation*}
for $\mathcal{H}$ the set of 1-Lipschitz functions and $Y$ the random variable
following the target distribution. We recall the lemma 4.2 of
\cite{chatterjee2008new} which provides with a standard implementation of the
Stein's method for this probabilistic distance with respect to the normal
distribution $\Normal (0,1)$.

\begin{lemma}[Normal approximation]
  \label{lem-discrete-mall:stein-normal}
  Let $L^\dag h(x) := h'(x) - x h(x)$. Then,
  \begin{equation}
    \label{eq-discrete-mall::stein-normal-bound}
    d_W (\mathcal{L}(F(X)), \Normal(0,1)) \le \sup_{\varphi \in \mathcal{H}_{*}} \left|\E [L^\dag  \varphi(F(X))] \right|,
  \end{equation}
  where
  $\mathcal{H}_* := \{h \in C^2(\R, \R): \: \|h'\|_{\infty} \le \sqrt{\frac{2}{\pi}}, \|h''\|_\infty \le 2\}$.
\end{lemma}
% \begin{proof}[Proof of lemma~\ref{lem-discrete-mall:stein-normal}]
%   We know that $\Normal(0,1)$ is the stationary distribution of an
%   Ornstein-Uhlenbeck process $(X^\dag (t))_{t \ge 0}$. It is the solution of
%   the stochastic differential equation
%   $\dif X^\dag (t) = -X^\dag (t) \dif t + \sqrt{2} \dif B(t)$ with
%   $X^\dag(0) = x$ for $x \in \R$ where $(B(t))_{t \ge 0}$ is a standard
%   Brownian motion. Let $Y$ a random variable following the law $\Normal(0,1)$.
% %   \todo[]{infinitesimal generator}
%   Let the semigroup $(P_t^\dag )_{t \ge 0}$ associated to the Markov process
%   $X^\dag$, then:
%   \begin{equation}
%     \E [h(Y)] - h(x) = \int_{0}^{+\infty} \frac{\dif}{\dif t} P_t^\dag h(x) \dif t = \int_{0}^{+\infty} L^\dag P_t^\dag h(x) \dif t
%   \end{equation}
%   and
%   \begin{equation}
%     |\E[h(F(X))] - \E [h(Y)]| = \left| \E \left[L^\dag \left(\int_{0}^{+\infty}  P_t^\dag h \dif t \right)(x) \right] \right|.
%   \end{equation}
%   We conclude by noting that
%   $\{x \longmapsto \int_{0}^{+\infty} P_t^\dag h(x) \dif t: \: h \in \mathcal{H}\} \subseteq \mathcal{H}_{*}$.
% \end{proof}
%
% This formula is also known as the semigroup method or the smart-path formula
% in the Stein's method literature \cite{DecreusefondSteinDirichletMalliavinmethod2015}.
In the following, we denote $d_W (\mathcal{L}(F(X)), \Normal(0,1))$ by
$d_{W}(F, \Normal (0, 1))$. For sake of conciseness, we denote by $\deltaap F$ the quantity $\deltaa F(X, X'_a)$.

\subsection{Rates in Lyapunov's conditional central limit}
\label{sec:lyapunov}
% The application of Malliavin Stein's method in our framework leads to a
% Lyapunov's quantitative conditional central limit theorem.

% \todo[]{Correct Stein without the order of derivative 2 no need for $h''$,
% gain an order in one-dimensional}
\begin{lemma}
  \label{lem-discrete-mall:bounds-wass-gradient}
  For any $F \in \cyl$ such that $\excond{F}{Z} = 0$. Then,
  \begin{multline}
    d_{W}(F, \Normal (0, 1)) \le \sup_{\psi \in \mathcal{H}_*} \left|\E \left[\sum_{a \in A} \psi (F(X^{\{a\}},X'_a)) \deltaap F D_a(-\mathsf{L}^{-1} F) - \psi(F)\right]\right| \\
    + \sum_{a \in A} \E [(\deltaap F)^2 |D_a \mathsf{L}^{-1} F|].
  \end{multline}
\end{lemma}

\begin{proof}[Proof of lemma \ref{lem-discrete-mall:bounds-wass-gradient}]
  We compute:
  $$\sup_{f^{\dag} \in \mathcal{H}_* } |\E [F (f^{\dag})(F)-  (f^{\dag})'(F)]|.$$
  Since $F$ is centered,
  \begin{align*}
    \E [F (f^{\dag})(F)] &= \E [\mathsf{L} (\mathsf{L}^{-1} F) f^{\dag}(F)] \\
                         &= -\sum_{a \in A} \E [D_a \mathsf{L}^{-1}F D_a f^{\dag}(F)]\text{ by integration by parts } \\
                         &= -\sum_{a \in A} \E \left[D_a \mathsf{L}^{-1}F \excond{(f^{\dag})'(F) - f^{\dag}(F^{\{a\}'})}{X,Z}\right] \\
                         &= -\sum_{a \in A} \E [D_a \mathsf{L}^{-1}F \deltaap f^{\dag}(F)].
                         % \text{ because }\indep{X'_a }{X_a} \text{ given
                         % }\sigma(X_b, b \ne a).
  \end{align*}
  Then, we use the Taylor expansion taking the reference point to be
  $F^{\{a\}'}$ instead of $F$, for all $a \in A$ yielding:
  \begin{align*}
    \deltaap f^{\dag}(F)&= f^{\dag}(F) - f^{\dag}(F^{\{a\}'}) \\
                              &= (f^{\dag})'(F^{\{a\}'}) \deltaap F + R_a, \\
  \end{align*}
  with
  $|R_a| \le \frac{\|(f^{\dag})''\|_{\infty}}{2} (\deltaap F)^2 = (\deltaap F)^2 $.
  Then,
  \begin{multline*}
    |\E [F f^{\dag}(F)-  (f^{\dag})'(F)]| \\
    \begin{aligned}
      & \le \left|\E \left[\sum_{a \in A}\deltaap F (D_a(-\mathsf{L}^{-1} F)) \left((f^{\dag})'(F^{\{a\}'})  - (f^{\dag})'(F) \right) \right] \right| \\
      &+ \sum_{a \in A} \E [(\deltaap F)^2 |D_a \mathsf{L}^{-1} F|].
    \end{aligned}
  \end{multline*}
  Because $(f^{\dag})''$ has Lipschitz-constant equal to $2$, we get the result.
\end{proof}

% That lemma is an extension of the lemma in \cite{decreusefond:hal-01565240}.
We prove a quantitative Lyapunov's conditional central limit theorem for random
variables with moments of order 3.
\begin{corollary}[Lyapunov's conditional central limit theorem]
  \label{cor-discrete-mall:lyapunov}
  Let $(X_n)_{n \in \N}$ be a sequence of thrice integrable, conditionally
  independent random variables given a latent random variable $Z$. Let us
  observe that
  \begin{equation*}
    \sigma_{j, Z}^2 = \var (X_j|Z), \: s_{n, Z}^2 = \sum_{j=1}^n \sigma_{j, Z}^2 \text{ and } \bar{X}_n = \frac{1}{s_{n, Z}} \sum_{j=1}^n \left(X_j - \excond{X_j}{Z}\right).
  \end{equation*}
  Then,
            %             \todo[]{Apply to Gaussian with variance $Z$ or $Z^2$,
            %             that explains why conditional distance might be better
            %             suited}
  \begin{equation}
    d_W (\bar{X}_n, \Normal (0, 1)) \le  2 (\sqrt{2} + 1)  \E \left[ \frac{1}{s_{n, Z}^3}\sum_{i=1}^n \left|X_i - \excond{X_i}{Z} \right|^3 \right].
  \end{equation}
            %             and for $N \sim \Normal (0, 1)$ independent of
            %             $\bar{X,Z}_n$ and $Z$:
            %             \begin{equation}
            %             \sup_{\phi \in \textnormal{Lip}_1} \left|\excond{\phi(\bar{X,Z}_n) - \phi(N)}{Z}] \right| \le \frac{2 (\sqrt{2} + 1)}{s_{n, Z}^3}\sum_{i=1}^n \E \left[\left|X_i - \excond{X_i}{Z} \right|^3 \, \Big\vert \, Z \right] .
            %             \end{equation}
\end{corollary}
% Kolmogorov That particular strategy does not lead bounds in Kolmogorov
% distance as it expresses the bound in function of the product of gradient and
% the difference operator.
The proof of the corollary follows the same steps as the one of \cite[Corollary
5.11]{decreusefond:hal-01565240}, using
lemma~\ref{lem-discrete-mall:bounds-wass-gradient}.

\begin{example}[Conditional Bernoulli random variables]
  Let $(U_i)_{i \in \N}$ independent uniform random variables, and
  $X_i = \ind{U_i \le Z}$, with $Z$ an arbitrary random variable lying in
  $[0, 1]$, then $(X_i)_{i \in \N}$ forms a sequence of conditionally
  independent random variables given $Z$.
  The law of $\mathcal{L}(X_i|X^{\{i\}}, Z)$ is a Bernoulli law of parameter
  $Z$. We compute the right-hand side of the Lyapunov theorem in this case.
  \begin{align*}
    s_{n,Z}^2 &= n Z (1 - Z) \\
    \excond{|X_i - \excond{X_i}{Z}|^3}{Z} &= Z(1-Z)(1 - 2Z).
  \end{align*}
  Hence,
  \begin{equation*}
    d_W (\bar{X}_n, \Normal (0, 1)) \le 2 (\sqrt{2} + 1) \E \left[\frac{1 - 2Z + 2Z^2}{\sqrt{Z(1 - Z)}}\right]n^{-1/2}.
  \end{equation*}

\end{example}

\subsection{Abstract bounds for U-statistics}

% Moreover \todo[]{Transition U-stats}
The chaos decomposition has a natural interpretation as a decomposition in terms
of degenerate U-statistics.

% It is reminiscent of the Hoeffding decomposition and can be seen as a
% decomposition in terms of degenerate U-statistics.
%
\begin{definition}[U-statistic \cite{hoeffding1948class}]
  % For an integer $p$, let $h: \prod_{} \to \R$ be a symmetric function, and
    %  $(X_1, \ldots, X_n)$ be a vector of $n$ real-valued independent random variables.
    Let a family of measurable functions
    \begin{math}
      h_I : E_I \to \R.
    \end{math}
    A U-statistic of degree (or order) $p$ is defined for any $n \ge p$ by:
    \begin{equation*}
        U = \sum_{I \in (A, p)} h_I(X_I) = \sum_{I \in (A, p)} W_I.
    \end{equation*}
    % A \textit{weighted U-statistic} of order $p$ has the form:
    % \begin{equation*}
    %     U_n = \sum_{B \in ([n], p)} w (B) h(X_B),
    % \end{equation*}
    % where the weight map $w: \: \llbracket 1, n \rrbracket^p \to \R$ is symmetric, and vanishes on the diagonal.
\end{definition}

\begin{definition}[Degenerate U-statistic]
    A degenerate U-statistic of order $p > 1$ is a U-statistic of order $p$
    % \begin{equation*}
    %   U = \sum_{I \in (A, p)}h_I(X_I)
    % \end{equation*}
    such that $\excond{h_I(X_{I }^{\{a\}}, x_a)}{Z} = 0$, for all $a \in A$ and $x_a \in E_a$.
\end{definition}
% One can decompose a degenerate U-statistic $F$ of order $p$ as:
% \begin{equation*}
%   F = \sum_{I \subset A, \, |I| = p} W_I \textnormal{ with }W_I = h_I(X_I).
% \end{equation*}
%
The space of degenerate U-statistics is exactly $\mathfrak{C}_p$.
Since we consider functionals given $Z$ hereafter, $h_I$ may be $\sigma(Z)$-measurable as well.

A convenient assumption in the proofs of quantitative limit theorems is the diffusiveness of the Markov generator at hand $L$, i.e. the associated carré du champ $\Gamma_L$ satisfies for $(F, G)$ in a dense algebra of $\dom{L}$:
\begin{equation*}
  \Gamma_L (\phi(F), G) = \phi'(F) \Gamma_L (F, G).
\end{equation*}
Due to the discreteness of the Malliavin structure, the operator $\mathsf{L}$ is not diffusive, but it is close to. We devise the following pseudo chain rule.
\begin{lemma}[First pseudo chain rule]
  \label{lem-discrete-mall:approx-diffusion}
   Let $\psi \in C^1(\R, \R)$. Let $G\in \mathcal{A}$ and $F \in L^2(E_A)$ such that $\psi(F) \in \mathcal{A}$, then:
  \begin{equation}
      \Gamma (\psi (F), G) = \frac{1}{2} \sum_{a \in A} \psi'(F) \excond{(\deltaap F) (\deltaap G)  }{X, Z} + R_{\psi} (F, G),
  \end{equation}
  where:
  $$|R_{\psi} (F, G) |\le \frac{\|\psi''\|_{\infty}}{4} \sum_{a \in A} \excond{|\deltaap G |(\deltaap F)^2}{X, Z}.$$
\end{lemma}

\begin{proof}[Proof of lemma~\ref{lem-discrete-mall:approx-diffusion}]

  % \begin{equation*}
  %   \Gamma (\psi (F), G) = \frac{1}{2} \sum_{a \in A} \excond{(\psi (F^{\{a\}'}) - \psi(F)) (G^{\{a\}} - G)}{X, Z}.
  % \end{equation*}
  We write the Taylor expansion of $\psi$, and:
  \begin{multline*}
    \excond{(\psi (F^{\{a\}'}) - \psi(F)) (G^{\{a\}'} - G)}{X,Z} =  \excond{\psi'(F) (\deltaap F) (\deltaap G) }{X,Z} \\
      + \excond{(G^{\{a\}'} - G) r_{\psi} (F, F^{\{a\}'} - F)}{X,Z}. \\
  \end{multline*}
Then,
  \begin{multline*}
    2\Gamma (\psi (F), G) =\psi'(F) \sum_{a \in A} \excond{(\deltaap F) (\deltaap G)}{X,Z}  \\
      + \sum_{a \in A} \excond{(G^{\{a\}'} - G) r_{\psi} (F, F^{\{a\}'} - F)}{X,Z} \\
  \end{multline*}
  where:
  $$r_{\psi} (x, y) = \psi(x+y) - \psi(x) - \psi'(x) y = \int_0^y (y - s) \psi''(x+s) \dif s.$$
  We note that $r_{\psi}$ satisfies:
  $$|r_{\psi} (x, y)| \le \frac{\|\psi''\|_{\infty}}{2} y^2,$$
  and we obtain the bound on the remainder.
\end{proof}

\begin{theorem}[Bounds in 1-Wasserstein distance]
  \label{thm-discrete-mall:bounds-wass-general}
  Assume that $F \in L^3(E_A)$, such that $\excond{F}{Z} = 0$ and $\E [F^2] = 1$, then we get the bound:
  \begin{multline}
    d_W (F, \Normal(0, 1)) \le \sqrt{\frac{2}{\pi}} \E |\Gamma (F, -\mathsf{L}^{-1} F) - 1| \\
    + \frac{1}{2} \sum_{a \in A} \E [|\deltaap \mathsf{L}^{-1} F|(\deltaap F)^2].
  \end{multline}
  Moreover, if $F \in L^4(E_A)$, then one has the further bound:
  \begin{multline}
    \label{eq-discrete-mall:bound-thm-1-Wasserstein}
    d_W (F, \Normal(0, 1)) \le \sqrt{\frac{2}{\pi}} \sqrt{\var (\Gamma (F, \mathsf{L}^{-1} F))} \\ + \frac{\sqrt{2}}{2}  \sqrt{-\E [F \mathsf{L} F ]} \sqrt{\sum_{a \in A} \E [|\deltaap F|^4 ]}.
\end{multline}
  % \begin{equation}
  %     d_W (F, \Normal(0, 1)) \le \sqrt{\frac{2}{\pi}} \sqrt{\var (\Gamma (F, \mathsf{L}^{-1} F))}  + \frac{\sqrt{2}}{2}  \sqrt{-\E [F \mathsf{L} F  )]} \sqrt{\sum_{a \in A} \E [|\deltaap F|^4 ]}.
  % \end{equation}
\end{theorem}

\begin{proof}[Proof of theorem \ref{thm-discrete-mall:bounds-wass-general}]
  We have:
  \begin{equation}
    \label{eq-discrete-mall::stein-eq-gen-carre}
    \begin{aligned}
        \E [L^{\dag} f^{\dag} (F)] &=  \E [F(f^{\dag})'(F) - (f^{\dag})''(F)]\\
        &= \E [\mathsf{L} \mathsf{L}^{ - 1} F (f^{\dag})'(F) ] - \E [(f^{\dag})''(F)] \\
        &= \E [\mathsf{L}^{ - 1} F \mathsf{L} ((f^{\dag})'(F) )] - \E [(f^{\dag})''(F)] \\
        &= \E [\Gamma (\mathsf{L}^{-1} ((f^{\dag})'(F) ), -\mathsf{L}^{- 1} F)] - \E [(f^{\dag})''(F)]
    \end{aligned}
  \end{equation}
  by integration by parts.
We use lemma~\ref{lem-discrete-mall:approx-diffusion} and obtain that:
\begin{equation*}
    \E [\Gamma (\mathsf{L} ((f^{\dag})'(F) ), -\mathsf{L}^{- 1} F)] \le \E [(f^{\dag})''(F) \Gamma (F, -\mathsf{L}^{ - 1} F)] + \E [R_{(f^{\dag})^{(3)}} (F, -\mathsf{L}^{ - 1} F)].
\end{equation*}
Thus,
$$\E [L^{\dag} f^{\dag} (F)] \le \sqrt{\frac{2}{\pi}} \E |\Gamma(F, -\mathsf{L}^{ -1} F) - 1| + \frac{1}{2} \sum_{a \in A} \E [|\deltaap \mathsf{L}^{-1} F| (\deltaap F)^2].$$
By Jensen's inequality for the first term and Cauchy-Schwarz inequality (for expectation of sum of random variables) for the second one, then by integration by parts, it yields:
\begin{multline*}
    \E [L^{\dag} f^{\dag} (F)] \le
    \sqrt{\frac{2}{\pi}} \sqrt{\var (\Gamma(F, \mathsf{L}^{ - 1} F))} \\
    + \frac{1}{2} \sqrt{\sum_{a \in A} \E [| \deltaap \mathsf{L}^{-1} F|^2]} \sqrt{\sum_{a \in A}  \E [(\deltaap F)^4]},
\end{multline*}
and the proof is complete.
\end{proof}

\begin{corollary}
  \label{cor-discrete-mall:bound-multi-chaos}
  If $F = \sum_{p=1}^m F_p$ is four times integrable functional where $F_p \in \ker(\mathsf{L} + p \textnormal{Id})$, then:
        \begin{multline}
          d_W (F, \Normal (0, 1)) \le \sqrt{\frac{2}{\pi}} \sum_{p, q =1}^m \frac{1}{q} \sqrt{\var \left[ \Gamma (F_p, F_q) \right]} \\
            + \sqrt{2}\sum_{p=1}^m \frac{1}{p} \sqrt{\E [F_p^2]} \left\{\sum_{p=1}^m p^{1/4} \left(  \sum_{a \in A} \E \left|\deltaap F \right|^4 \right)^{1/4} \right\}^2.
        \end{multline}
  %   Moreover,
  %   \begin{equation}
  %     \begin{aligned}
  %          &d_{Kol} (F, \Normal (0, 1)) \le  \sqrt{\var \left[ \sum_{p, q =1}^m \frac{1}{q}\Gamma (F_p, F_q) \right]} \\
  %          +& \left(\frac{\sqrt{\pi}}{2} (-\sum_{p=1}^m p\E [F^2])^{1/2} + \frac{1}{2} \left( \E [F^4] \E \left[\sum_{a \in A} ((\mathsf{L}^{-1} F)^{\{a\} '} - \mathsf{L}^{-1} F)^4 \right] \right)^{1/4}\right)   \left(\sum_{a \in A} \E [|\deltaap F|^4 ] \right)^{1/2}
  %          +& \\
  %     \end{aligned}
  %  \end{equation}
\end{corollary}

\begin{proof}[Proof of corollary~\ref{cor-discrete-mall:bound-multi-chaos}]
  We use the decomposition of $\mathsf{L}^{-1}$ as to develop the first and second terms in \eqref{eq-discrete-mall:bound-thm-1-Wasserstein}.
  The final result is obtained after using Cauchy-Schwarz inequality.
\end{proof}
%
% Using the carré du champ operator boils down to proceed with the difference operator for the computations of the bounds.
%
That is the starting point towards a partial fourth moment limit theorem.
% The underlying Dirichlet structure has been quite of interest in the field of concentration of measure, but not in Malliavin calculus.
%

% \todo{multivariate possible but tedious computations}

\subsection{Fourth moment phenomenon}
\label{sec:partial-fourth-moment}

We adapt the proof of \cite{azmoodeh2014fourth}, requiring a second pseudo chain rule that expresses the carré du champ operator as an approximation of a derivation operator in its two arguments.

\begin{lemma}[Second pseudo chain rule]
  \label{lem-discrete-mall:approx-diffusion-bis}
  Let $\varphi, \psi$ be twice differentiable functions such that their second derivative is bounded Lipschitz-continuous.
  Assume that $F$ a four times integrable functional such that $\varphi(F) \in \mathcal{A}$, $F \in \mathcal{A}$ and $\excond{F}{Z} = 0$, then one has:
  % \begin{equation}
  %   \begin{aligned}
    \begin{multline}
       \Gamma(\varphi (F), \psi(F)) = (\varphi'\psi')(F) \Gamma (F, F)
      \\ - \frac{1}{4} (\varphi'' \psi' + \varphi' \psi'') (F) \sum_{a \in A} \excond{(\deltaap F )^3}{X,Z}
      + \sum_{a \in A} R_a,
    \end{multline}

  %   \end{aligned}
  % \end{equation}
  with:
  \begin{equation*}
      R_a = \frac{1}{2} \left(\excond{R_{a, \varphi \psi}^{(4)}(F)}{X,Z}- \varphi(F) \excond{R_{a,\psi}^{(4)}(F)}{X,Z} - \psi(F) \excond{R_{a,\varphi}^{(4)}(F)}{X,Z}\right)
  \end{equation*}
  and:
    \begin{equation*}
        R_{a,\psi}^{(4)} \le \frac{\|\psi^{(4)}\|_{\infty}}{24} \excond{(\deltaap F)^4}{X,Z} \text{ for any }\psi \text{ fourth times differentiable}.
    \end{equation*}
\end{lemma}

\begin{proof}[Proof of lemma~\ref{lem-discrete-mall:approx-diffusion-bis}]
  % \begin{equation}
  %     \begin{aligned}
    We have:
    \begin{multline}
      2\Gamma (\varphi (F), \psi(F)) \\
      \begin{aligned}
        &= 2 \varphi'(F) \psi'(F) \Gamma (F, F) - \frac{3}{6} (\varphi'' \psi' + \varphi' \psi'') (F) \sum_{a \in A} \excond{(\deltaap F)^3}{X,Z} \\
          &+ \sum_{a \in A} \excond{R_{a, \varphi \psi}^{(4)}(F) -  \varphi(F) R_{a,\psi}^{(4)}(F)- \psi(F) R_{a,\varphi}^{(4)}(F)}{X,Z},
      \end{aligned}
    \end{multline}
  %     \end{aligned}
  % \end{equation}
  with:
  \begin{equation*}
      R_{a,\phi}^{(4)} = \frac{1}{6}\excond{ \int_{F}^{F^{\{a\}'}} \phi^{(4)}(x) (x - F)^4 \dif x}{X,Z},
  \end{equation*}
  for $\phi$ a four times differentiable function.
\end{proof}
% That is the property that replaces the diffusion assumption in \cite{azmoodeh2014fourth}. We adopt the strategy of the paper, using our chaos decomposition.
We focus on functionals in the $p$-th chaos for $p > 0$, as to obtain such kind of bound:
$$\var [\Gamma (F, F)] \le C (\E [F^4] - 3 \E [F^2]^2) + \text{remainder}.$$

\begin{lemma}
    \label{lem-discrete-mall:expectation-square-ineq}
    Let $G \in \oplus_{k=0}^q \mathfrak{C}_k$. Then for any $\eta \ge q$,
    \begin{equation}
        \E [G (\mathsf{L} + \eta \text{Id})^2 G] \le \eta \E [G (\mathsf{L} + \eta \text{Id}) G] \le  c \E [G (\mathsf{L} + \eta \text{Id})^2 G],
    \end{equation}
    where
    $$c = \frac{1}{\eta - q} \wedge 1.$$
\end{lemma}

\begin{proof}[Proof of lemma~\ref{lem-discrete-mall:expectation-square-ineq}]
  Since $G \in  \oplus_{k=0}^q \mathfrak{C}_k$, we write
  \begin{equation}
    G = \sum_{k=0}^q \pi_k(G) \text{ and } \mathsf{L} G = -\sum_{k=0}^q k \pi_k(G).
  \end{equation}
  It follows that
  \begin{align*}
    \E [G (\mathsf{L} + \eta \text{Id})^2 G] &= \E [G \mathsf{L} (\mathsf{L} + \eta \text{Id}) G] + \eta \E [G(\mathsf{L} + \eta \text{Id}) G] \\
    &= \E [G \sum_{k=0}^q k (k - \eta) \pi_k(G)] + \eta \E [G (\mathsf{L} + \eta \text{Id}) G].\\
  \end{align*}
  By orthogonality of the chaos,
  \begin{equation*}
    \E [G \sum_{k=0}^q k (k - \eta) \pi_k(G)] = -\E [\sum_{k=0}^q k (\eta -k) \pi_k(G)^2] \le 0,
  \end{equation*}
  and the inequality holds in view on the assumption on $\eta$. In the same vein,
  \begin{align*}
    \E [G (\mathsf{L} + \eta \text{Id}) G] &= \sum_{k=0}^q (\eta - k) \E [\pi_k(G)^2] \\
    &\le c \sum_{k = 0}^q (\eta - k)^2 \E [\pi_k(G)^2] \\
    &= c \E [G(\mathsf{L} + \eta \text{Id})^2].
  \end{align*}
  Thus, it yields the result.
\end{proof}

\begin{lemma}
  \label{lem-discrete-mall:ipp-trinome}
  For $F \in \mathfrak{C}_p \cap L^4(E_A)$ and $Q$ a polynomial of degree two and $a > 0$,
  \begin{equation}
      \E [Q(F) (\mathsf{L} + a p \textnormal{Id}) Q(F) ] = p \E \left[a Q^2 (F) - \frac{Q'(F) F}{3Q''(F)}\right] - \E [R_Q(F)],
  \end{equation}
  where $R_Q$ is a remainder term that depends on $Q$.
  For $Q = H_2 = X^2 - 1$ the second Hermite polynomial, the remainder reads off:
    \begin{equation}
       \E [R_{Q}] = \E [R_{H_2}] = \frac{1}{6} \E \left[\sum_{a \in A} |\deltaap F|^4 \right].
    \end{equation}
\end{lemma}

\begin{proof}[Proof of lemma~\ref{lem-discrete-mall:ipp-trinome}]
  We first integrate by parts, then use the pseudo chain rule of lemma~\ref{lem-discrete-mall:approx-diffusion-bis}:
  \begin{equation}
      \begin{aligned}
          \E [Q(F) \mathsf{L} Q(F)] &= - \E[\Gamma (Q(F), Q(F))] \\
          &=  -\E[Q'(F)^2 \Gamma (F, F)] \\
          &+ \frac{1}{6} (Q^2)^{(3)} (F)  \sum_{a \in A} \excond{(\deltaap F)^3}{X,Z}  \\
          &-\frac{1}{2} \sum_{a \in A} \E \left[\excond{R_{a, Q^2}^{(4)}(F)}{X,Z}- 2 Q(F) \excond{R_{a,Q}^{(4)}(F)}{X,Z} \right].\\
      \end{aligned}
  \end{equation}
  Since $Q^{(3)} = 0$, we have:
  \begin{equation}
      \begin{aligned}
          \E [Q(F) \mathsf{L} Q(F)] &= -\E \left[[Q'(F)^2 \Gamma (F, F) \right] \\
          &+ \frac{1}{6} \E \left[(Q^2)^{(3)} (F)  \sum_{a \in A} \excond{(\deltaap F)^3}{X,Z} \right]  \\
          &- \frac{1}{2}\sum_{a \in A} \E \left[\excond{R_{a, Q^2}^{(4)}(F)}{X,Z} \right].
      \end{aligned}
  \end{equation}
  Moreover,
  \begin{equation}
      \left(\frac{Q'(F)^3}{3 Q''(F)}\right)' = \frac{3 Q'(F) Q''(F)^2}{3 Q''(F)^2} = Q'(F)^2.
  \end{equation}
  Subsequently, we use the pseudo chain rule of lemma~\ref{lem-discrete-mall:approx-diffusion-bis} taking $\psi = \text{Id}$ and $\varphi = \frac{Q'(\cdot)^3}{3 Q''(\cdot)} $:
  \begin{equation}
      \begin{aligned}
          \E[Q'(F)^2 \Gamma (F, F)]  &= \E \left[\Gamma \left(\frac{Q'(F)^3}{3 Q''(F)}, F \right)\right] \\
          &+ \frac{1}{4} \E\left[(\varphi'' \psi' + \varphi' \psi'') (F)  \sum_{a \in A} \excond{(\deltaap F)^3}{X,Z}\right]  \\
          &- \sum_{a \in A} \E \left[ \excond{R_{a, \varphi \psi}^{(4)}(F)}{X,Z} - \varphi(F) \excond{R_{a,\psi}^{(4)}(F)}{X,Z} \right] \\
          &- \E \left[  F \excond{R_{a,\varphi}^{(4)}(F)}{X,Z} \right] \\
          &= \E \left[\Gamma \left(\frac{Q'(F)^3}{3 Q''(F)}, F \right)\right] + \frac{1}{4} \E \left[(Q'(\cdot)^2)'(F)  \sum_{a \in A} (\deltaap F)^3\right]\\
          &- \sum_{a \in A} \frac{1}{2} \E \left[ R_{a, \varphi \psi}^{(4)}(F)  - F R_{a,\varphi}^{(4)}(F) \right].
      \end{aligned}
  \end{equation}
  Finally,
  \begin{equation}
      \begin{aligned}
          \E [Q(F) \mathsf{L} Q(F)] &= -\E \left[\Gamma \left(\frac{Q'(F)^3}{3 Q''(F)}, F \right)\right] \\
          &+ \E \left[\left(\frac{1}{4} (Q'(\cdot)^2)'(F) - \frac{1}{12} (Q^2)^{(3)} (F) \right)  \sum_{a \in A} (\deltaap F)^3\right]\\
          &+ \frac{1}{2} \sum_{a \in A} \E \left[ R_{a, \varphi \psi}^{(4)}(F) - R_{a, Q^2}^{(4)}(F) - F R_{a,\varphi}^{(4)}(F) \right] \\
          &= -\E \left[\Gamma \left(\frac{Q'(F)^3}{3 Q''(F)}, F \right)\right] \\
          &+ \E \left[ \left(\frac{1}{4} (Q'(\cdot)^2)'(F) - \frac{1}{12} (Q^2)^{(3)} (F) \right)  \sum_{a \in A} (\deltaap F)^3\right]  \\
          &+ \frac{1}{2}\sum_{a \in A} \E \left[R_{a, \varphi \psi}^{(4)}(F)- R_{a, Q^2}^{(4)}(F) \right].
      \end{aligned}
  \end{equation}
  Because $F \in \mathfrak{C}_p$, we have: $-\E \left[\Gamma \left(\frac{Q'(F)^3}{3 Q''(F)}, F \right)\right] = \E [\frac{Q'(F)^3}{3 Q''(F)} \mathsf{L} F ] = -p \E \left[\frac{Q'(F)^3}{3 Q''(F)} F \right]$.
  For $Q = H_2 = X^2 - 1$ the second Hermite polynomial,
\begin{equation*}
    \frac{Q'(F)^3}{3 Q''(F)} = \frac{4}{3} X^3,
\end{equation*}
so $\left(\frac{Q'(\cdot)^3}{3 Q''(\cdot)} \cdot \right)^{(4)} = 32$ and $(Q^2)^{(4)} = 24$.
Thus,
\begin{equation}
    \sum_{a \in A} \E \left[ R_{a, \varphi \psi}^{(4)}(F) - R_{a, Q^2}^{(4)}(F)\right] = \frac{(32 - 24)}{24} \sum_{a \in A}  \E \left[
    |\deltaap F|^4\right].
  \end{equation}
  Since $(Q'(\cdot)^2)'(F) = 8 F$, and $(Q^2)^{(3)}(F) = 24 F$, the result follows.
\end{proof}

The assumption under which a fourth moment theorem holds, is that $F \in \mathfrak{C}_p$ is a chaos eigenfunction with respect to the Markov generator $\mathsf{L}$ i.e.:
\begin{equation}
  \label{eq-discrete-mall:eigenfunction}
  F^2 \in \oplus_{k=0}^{2p} \mathfrak{C}_k.
  \tag{EGF}
\end{equation}
It is analog to the one in \cite{ledoux2012chaos,azmoodeh2014fourth}.
We show that it holds for an important class of U-statistics, homogeneous sums.
%
% \todo[]{Notation and example out of nowhere}
We shall use the notation $(A, p)$ that stands for
the set of p-tuples of distinct elements of $A$.

\begin{example}[\textbf{Conditionally independent homogeneous sums}]
Let $p > 0$. If there exists $(a_I)_{I \subset A} \in \R^{\mathcal{P}(A)}$
such that
\begin{equation}
  \label{eq-discrete-mall:homogeneous-sum}
  W= \sum_{k=1}^p \sum_{I \in (A, k)} a_I \prod_{i \in I} X_i,
\end{equation}
then
% \todo[]{Remove or shorten one item}
\begin{enumerate}
  \item $W$ is square-integrable homogeneous sum of order $p$ if $X_i$ are
        $2p$-integrable. In that case, $W\in \cyl$.
  \item \begin{equation*} \excond{W}{Z} = \sum_{k=1}^p \sum_{I \in (A, k)} a_I \prod_{i \in I} \excond{X_i}{Z}
  \end{equation*}
        is a homogeneous sum of random variables $\hat{X}_i = \excond{X_i}{Z}$
        for $i \in I$ with $I \in (A, k)$ for $k \le p$.
\end{enumerate}
Remark that $(a_I)_{I \subset A}$ may be a sequence of random variables, in
which case there exists a family of functions $(g_I)_{I \subset A}$ such that
$a_I = g_I(Z)$.
\end{example}

\begin{lemma}
  \label{lem-discrete-mall:homogeneous-sum-eigenfunction}
  Let $W$ a homogeneous sums of conditionally independent random variables given $Z$. Then \eqref{eq-discrete-mall:eigenfunction} holds.
\end{lemma}

\begin{proof}[Proof of lemma~\ref{lem-discrete-mall:homogeneous-sum-eigenfunction}]
  % First, we can consider the case where $\hat a_i = 0$ for all $i \in \N^p$.
  Let us denote by $W_I$ the component of $F$ in \eqref{eq-discrete-mall:homogeneous-sum} proportional to $\prod_{\alpha \in I} X_{\alpha}$. We want to prove that there exist $G_1, \ldots, G_{2p}$ with $G_i \in \mathfrak{C}_i \cup \{0\}$ such that $W_I W_J = \sum_{i=1}^{2p} G_i$.
  Note that if $I \cap J = \emptyset$, and $a \in I$, then $a$ is not in $J$ and vice versa. Therefore, $W_I W_J \in \mathfrak{C}_{|I| + |J|}$. In general,
%
%   If $I \cap J = \{a\}$, $\excond{W_I W_J}{\mathcal{G}_a} = \excond{X_a^2}{Z} W_{I \sminus \{a\}} W_{J \sminus \{a\}}$ and
%   \begin{equation*}
%     W_I W_J = X_a^2 W_{I \sminus \{a\}} W_{J \sminus \{a\}}
%   \end{equation*}
% where $W_{I \sminus \{a\}} W_{J \sminus \{a\}} \in \mathfrak{C}_{|I| + |J| - 2}$.
% \todo[]{to finish}
% We need $\excond{X_a^2}{Z} = 1$ to be centered.
\begin{align*}
  W_I W_J &\propto \prod_{\alpha \in I} Y_\alpha \prod_{\beta \in J} Y_\beta \\
  &= \prod_{\gamma \in (I \sminus J) \cup (J \sminus I)} Y_\gamma \prod_{\delta \in I \cap J} Y_\delta^2 \\
  &= \prod_{\gamma \in (I \sminus J) \cup (J \sminus I)} Y_\gamma \prod_{\delta \in I \cap J} (Y_\delta^2 - \excond{Y_\delta^2}{Z} + \excond{Y_\delta^2}{Z} ) \\
  &= \sum_{K \subset I \cap J} \prod_{\gamma \in (I \sminus J) \cup (J \sminus I)} Y_\gamma \prod_{\delta \in K} (Y_\delta^2 - \excond{Y_\delta^2}{Z}) \prod_{\delta \in (I \cap J) \sminus K} \excond{Y_\delta^2}{Z}.
\end{align*}
For $a \in A$:
\begin{multline*}
  \excond{\prod_{\gamma \in (I \sminus J) \cup (J \sminus I)} Y_\gamma \prod_{\delta \in K} (Y_\delta^2 - \excond{Y_\delta^2}{Z}) \prod_{\delta \in (I \cap J) \sminus K} \excond{Y_\delta^2}{Z}}{\mathcal{G}_a^Z} \\
  = \begin{dcases}
    & 0 \text{ if }a \in K \cup ((I \sminus J) \cup (J \sminus I)) \\
    & \prod_{\gamma \in (I \sminus J) \cup (J \sminus I)} Y_\gamma \prod_{\delta \in K} (Y_\delta^2 - \excond{Y_\delta^2}{Z}) \prod_{\delta \in (I \cap J) \sminus K} \excond{Y_\delta^2}{Z}\text{ otherwise}.
  \end{dcases}
\end{multline*}
Hence, we get
\begin{equation*}
  \prod_{\gamma \in (I \sminus J) \cup (J \sminus I)} Y_\gamma \prod_{\delta \in K} (Y_\delta^2 - \excond{Y_\delta^2}{Z}) \prod_{\delta \in (I \cap J) \sminus K} \excond{Y_\delta^2}{Z} \in \mathfrak C_{|K \cup (I \sminus J) \cup (J \sminus I)|}
\end{equation*}
 with $|K \cup ((I \sminus J) \cap (J \sminus I))| \le |I \cup J| \le 2p$.
Thus, \eqref{eq-discrete-mall:eigenfunction} holds.
% The same goes for the components if for all $i \in \N^p$, $a_i = 0$.
% In the general case, by linearity, theres is only the case $W_I W_J \propto \prod_{\alpha \in I} Y_\alpha \prod_{\beta \in J} \hat Y_\beta$ left. Since $(Y_\alpha, \ldots, \hat Y_\beta, \ldots)_{\alpha \in I, \beta \in J}$ forms a sequence of centered conditionally independent random variables $W_I W_J \in \mathfrak{C}_{|I| + |J|}$.
\end{proof}

\begin{proposition}
  \label{prop-discrete-mall:var-carre-champ-fourth}
    For $F \in \mathfrak{C}_p \cap L^2(E_A)$ such that $\E[F^2] = 1$ and \eqref{eq-discrete-mall:eigenfunction} holds, one has:
    \begin{equation}
      \label{eq-discrete-mall:var-carre-champ-fourth-and-remainder}
        \E [(\Gamma (F, F) - p)^2] \le \frac{p^2}{3} |\E [F^4] - 3| + \frac{p}{12} \E \left[\sum_{a \in A} |\deltaap F|^4 \right].
    \end{equation}
\end{proposition}

\begin{proof}[Proof of proposition~\ref{prop-discrete-mall:var-carre-champ-fourth}]
  By the very definition of $\Gamma$, one has:
  \begin{align*}
      \Gamma (F, F) - p &= \frac{1}{2} \mathsf{L} (F^2) - F\mathsf{L} F - p
      =  \frac{1}{2} \mathsf{L} (F^2) + p F^2 - p \text{ for }F \in \mathfrak{C}_p \\
      &= \frac{1}{2}(\mathsf{L} + 2p \text{Id}) (F^2 - 1).
  \end{align*}
  It follows that:
  $$ \E [(\Gamma (F, F) - p)^2] = \frac{1}{4} \E [\left((\mathsf{L} + 2p \text{Id}) (F^2 - 1)\right)^2]. $$
  Since $\mathsf{L}$ is a self-adjoint operator, this yields:
  \begin{equation*}
      \E [(\Gamma (F, F) - p)^2] = \frac{1}{4} \E [H_2(F)(\mathsf{L} + 2p \text{Id})^2 H_2(F)].
  \end{equation*}
  As \eqref{eq-discrete-mall:eigenfunction} holds, we are in position to apply lemma \ref{lem-discrete-mall:expectation-square-ineq} with $q = 2p$ and $\eta = 2p$:
  \begin{equation}
      \E [(\Gamma (F, F) - p)^2] \le \frac{p}{2} \E [H_2(F)(\mathsf{L} + 2p \text{Id}) H_2(F)].
  \end{equation}
  According to lemma \ref{lem-discrete-mall:ipp-trinome}, with $a = 2$,
  \begin{align*}
      \frac{p}{2}  \E [H_2(F)(\mathsf{L} + 2p \text{Id}) H_2(F)] &= \frac{p^2}{2} \E \left[2 (F^2 - 1)^2 - \frac{4}{3} F^4\right] + \frac{p}{2} \E [R_{H_2} (F)] \\
      &=  \frac{p^2}{6} \E \left[6 (F^2 - 1)^2 - 4 F^4\right] + \frac{p}{2} \E [R_{H_2} ] \\
      &= \frac{p^2}{3} \E [F^4 - 6 F^2 + 3] + \frac{p}{2} \E [R_{H_2} ].
  \end{align*}
  Thus, it yields
  \begin{equation}
      \E [(\Gamma (F, F) - p)^2] \le \frac{p^2}{3} |\E [F^4 - 6 F^2 + 3]| + \frac{p}{2} |\E [R_{H_2} ]|,
  \end{equation}
  and the proof is complete, using again lemma~\ref{lem-discrete-mall:ipp-trinome}.
\end{proof}

\subsection{Quantitative De Jong's theorems}

Many papers are devoted to find the optimal conditions for the asymptotic normality of U-statistics.
The criterion established in \cite{de1990central} is related to the fourth moment phenomenon. The extra assumption is a negligibility condition also known as the Lindeberg-Feller condition. Fix $A_m$ a finite subset of cardinal $m$ such that $F = F(X_{A_m})$ and $\E[F^2] = 1$, that means:
\begin{equation}
  \label{eq-discrete-mall:max-influence}
  \rho_{A_m}^2 = \max_{i \in A_m} \sum_{I \ni i, \, I \subseteq A_m, \,|I| = p} \E [W_I^2] \xrightarrow{m \to +\infty} 0.
\end{equation}
In some papers \cite{dobler2017quantitative}, the term $\rho_{A_m}$ is called maximal influence of the random variables on the total variance of the degenerate U-statistics $F$. In the following, we shall denote it by $\rho$.  The condition \eqref{eq-discrete-mall:max-influence} is not necessary for asymptotic normality to hold, but there exist counterexamples for which the sequence of fourth cumulants of functionals of independent Rademacher random variables converges to 0 while \eqref{eq-discrete-mall:max-influence} does not hold (see \cite{dobler2019quantitative}).
% \begin{definition}[Maximal influence]
%   Let $F$ a degenerate U-statistics of order $p$ such that:
%   \begin{equation}
%       F= \sum_{I \in (A, p)} W_I
%   \end{equation}
%   Let us call
%   \begin{equation*}
%       \rho_A^2 = \max_{i \in A} \sum_{I \ni i} \E [W_I^2].
%   \end{equation*}
% \end{definition}
%
% The next result shows the analogous phenomenon for degenerate U-statistics of conditionally independent random variables. The degeneracy holds whenever $\excond{W_I(X_{i_1}, x_{i_2}, \ldots, x_{i_p})}{Z} = 0$ for every $I  = \{i_1, \ldots, i_p\} \subseteq A$.
We show that the quantity is related to the remainder above.
% Thereafter, we may use the abbreviation "cnctd" that stands for connected.
%
\begin{definition}[Connectedness of subsets]
  The $r$-tuple $(I_1, \ldots, I_r)$ subsets of $A$ is connected if the intersection graph of $\{I_1, \ldots, I_r\}$ is connected, i.e.
  the graph $G$ with vertex set $\{I_1, \ldots, I_r\}$ and edge set $E(G) = \{\{I_i, I_j\}| \: i \ne j, I_i \cap I_j \ne \emptyset\}$ is connected.
\end{definition}
% \begin{remark}
%   In the case where $r = 4$, there are exactly six simple connected graphs with only four vertices (up to isomorphisms): linear (4-path), 4-star, square, kite (tadpole), diag (diamond), and 4-complete graph $K_4$.
% \end{remark}
\begin{lemma}%[Bound by maximal influence]
\label{lem-discrete-mall:bound-max-influence}
  If $F \in \mathfrak{C}_p \cap L^4(E_A)$, then:
  \begin{equation}
      \sum_{a \in A} \E [|\deltaap F|^4 ] \le 16 p \sum_{(I, J, K, L) \textnormal{ connected}} | \E [W_I W_J W_K W_L] |.
  \end{equation}
  Moreover, assuming the hypercontractivity condition, i.e.
  \begin{equation}
    \label{eq-discrete-mall:hypercontractivity}
      % \sup_{A} D_A < +\infty \text{ where }
     \sup_{J \in (A, p)} \frac{\E [W_J^4]}{\E[W_J^2]^2} < +\infty,
      \tag{HC}
  \end{equation}
  there exists a constant $c_p$ that depends only on $p$ such that:
  \begin{equation}
    \label{eq-discrete-mall:fourth-difference-maximal-influence}
      \sum_{a \in A} \E [|\deltaap F|^4] \le c_p \rho^2.
  \end{equation}
  % \chris{such that $F = \excond{F}{X_B}$.}
\end{lemma}

\begin{proof}[Proof of lemma~\ref{lem-discrete-mall:bound-max-influence}]
  Because $(a + b)^4 \le 8 (a^4 + b^4)$, one has:
  \begin{align*}
      \sum_{a \in A} \E \left|\deltaap F\right|^4
      &\le 8\sum_{a \in A} \E \left[\left(\sum_{I \ni a, |I| \le p} W_I^{\{a\}'}\right)^4 + \left(\sum_{I \ni a, |I| \le p} W_I \right)^4 \right] \\
      &= 16 \sum_{a \in A} \E\left[\left(\sum_{I \ni a, |I| \le p} W_I\right)^4 \right] \\
      &\le 16 \sum_{I \cap J \cap K \cap L \ne \emptyset} |I \cap J \cap K \cap L| \E [W_I W_J W_K W_L] \\
      &\le 16 p \sum_{I \cap J \cap K \cap L \ne \emptyset} | \E [W_I W_J W_K W_L] | \\
      &\le 16 p \sum_{I, J, K, L \text{ connected}} | \E [W_I W_J W_K W_L] |.
  \end{align*}
  Then, we bound it by the maximal influence, using the generalized Hölder inequality:
  \begin{align*}
      | \E [W_I W_J W_K W_L] | &\le \left(\E [W_I^4]\E [W_J^4]\E [W_K^4]\E [W_L^4] \right)^{1/4} \\
      &\le \max_{J \in A, |J| = p} \frac{\E [W_J^4]}{\E[W_J^2]^2} \left(\E [W_I^2]^2\E [W_J^2]^2\E [W_K^2]^2\E [W_L^2]^2\right)^{1/4} \\
      % &= \kappa^{(1)} \sigma_I \sigma_J \sigma_K \sigma_L,
  \end{align*}
  with $\sigma_I^2 = \E [W_I^2]$.
  % and $\kappa^{(1)} = \max_{J \in A, |J| = p} \frac{\E [W_J^4]}{\sigma_J^4}$.
  Then the proposition 2.9 of \cite{dobler2017quantitative} can be extended for functionals of conditionally independent random variables and implies that:
  \begin{equation*}
      \sum_{I \cap J \cap K \cap L \ne \emptyset} \sigma_I \sigma_J \sigma_K \sigma_L \le C_p \rho^2,
  \end{equation*}
  where the finite constant $C_p$ only depends on $p$. It yields the existence of $c_p > 0$ such that the inequality \eqref{eq-discrete-mall:fourth-difference-maximal-influence} holds true.
  % Thus, it yields that:
  % \begin{equation}
  %     \sum_{a \in A} \E \left|\deltaap F\right|^4 \le 16 D C_p p \rho_A^2.
  % \end{equation}
\end{proof}

We are now in position to state a partial fourth moment limit theorem.
\begin{theorem}[Quantitative De Jong's limit theorem I]
  \label{thm-discrete-mall:quant-dejong-conditional}
  Let $F \in L^4(E_A)$ a degenerate U-statistics of order $p$ of conditionally independent random variables such that $\excond{F}{Z}= 0$ and $\E [F^2] = 1$. If we suppose the hypercontractivity condition \eqref{eq-discrete-mall:hypercontractivity} and the assumption \eqref{eq-discrete-mall:eigenfunction}, then one has the bound:
  \begin{equation}
      d_W (F, \Normal(0, 1)) \le \sqrt{\frac{2}{3\pi}} \sqrt{|\E\left[ F^4 \right] - 3 |} + \tilde C_p\rho,
  \end{equation}
  with $\tilde C_p$ a positive constant that only depends on $p$.
%   Furthermore,
%   \begin{equation}
%     d_{Kol} (W, \Normal(0, 1)) \le \sqrt{\frac{1}{3}} \sqrt{\E\left[ W^4 \right] - 3 } + \tilde{C}_p\rho_A,
% \end{equation}
% with $\tilde C_p$ a positive constant that only depends on $p$ and the fourth moment of $W$.
\end{theorem}

\begin{proof}
  By corollary~\ref{cor-discrete-mall:bound-multi-chaos},
    \begin{equation*}
      d_W (F, \Normal (0, 1)) \le \sqrt{\frac{2}{\pi}}  \frac{1}{p} \sqrt{\var \left[ \Gamma (F, F) \right]} \\
        + \sqrt{2}  \sqrt{\E [F^2]}  \left(  \sum_{a \in A} \E \left[\left|\deltaap F \right|^4 \right] \right)^{1/2}.
    \end{equation*}
  The combination of \eqref{eq-discrete-mall:var-carre-champ-fourth-and-remainder} and lemma~\ref{lem-discrete-mall:bound-max-influence} yields the final upper bound.
\end{proof}
% \begin{remark}
 The upper bound of the remainder expressed in terms of maximal influence is not used in the subsequent applications, so we drop the \eqref{eq-discrete-mall:hypercontractivity} condition.
 % That approach only holds for functionals of independent Bernoulli random variables.
% \end{remark}

A related result to the fourth moment phenomenon appears in \cite{de1996central}. We prove the associated quantitative statement for functionals of conditionally independent random variables. We prepare the proof with the following proposition.
\begin{proposition}
  \label{prop-discrete-mall:var-carre-champ-connected}
  If $F = \sum_{p=1}^m F_p$ where $F_p = \sum_{|I| = p} W_I \in \mathfrak{C}_p$, assuming there exists $C \in \R^+$ such that for all $I, J \subset A$, and $a \in A$,
  that
  \begin{equation}
    \label{eq-discrete-mall:hypothesis1}
    \frac{\excond{W_I W_J}{\mathcal{G}^a}}{W_{I \sminus \{a\}} W_{J \sminus \{a\}}} < C \,\mathbb{P}\text{-a.s.},
    \tag{H1}
  \end{equation}
  then for $p \ne q$:
  \begin{equation}
      \sqrt{\var \left[\Gamma (F_p, F_q) \right]} \lesssim \sqrt{\sum_{\substack{(I, J, K, L) \textnormal{ connected}}} |\E [W_I W_J W_K W_L] |},
  \end{equation}
  for $I,J, K, L$ sets of size less than $\max(p, q)$.
\end{proposition}

\begin{proof}[Proof of proposition~\ref{prop-discrete-mall:var-carre-champ-connected}]
  The carré du champ reads for $p \ne q$:
  \begin{align*}
    \Gamma (F_p, F_q) &= \Gamma (\sum_{|I| = p} W_I, \sum_{|J| = q} W_J) \\
    &=  \sum_{|I|, |J| = p, q} \Gamma ( W_I, W_J).
  \end{align*}
  Hence,
  \begin{align*}
      2 \Gamma (F_p, F_q) &= \sum_{|I|, |J| = p, q}\left(\mathsf{L}(W_I W_J) + (p+q) W_I W_J\right) \\
      &=  \sum_{|I|, |J| = p, q} \left((p+q) W_I W_J - \sum_{a \in A} D_a (W_I W_J)\right) \\
      &=\sum_{|I|, |J| = p, q} \left((p+q) W_I W_J - \sum_{a \in I \cup J} D_a (W_I W_J)\right) \\
      &=  (p+q)  \sum_{\substack{|I|, |J| = p, q \\ I \cap J = \emptyset}} W_I W_J + \sum_{\substack{|I|, |J| = p, q \\ I \cap J \ne \emptyset}} (|I| + |J| - |I \cup J|) W_I W_J \\
      &+ \sum_{a \in I \cup J} \excond{W_I W_J}{\mathcal{G}_a}.\\
  \end{align*}
  Because of the spectral decomposition, $\excond{W_I}{\mathcal{G}_a} = 0$ for $a \in I$. Let $J$ such that $a \notin J$, then $\excond{W_I W_J}{\mathcal{G}_a} = W_J \excond{W_I}{\mathcal{G}_a} = 0$.
  \begin{equation*}
      2 \Gamma (F_p, F_q) = (p+q) \sum_{\substack{|I|, |J| = p, q \\ I \cap J = \emptyset}} W_I W_J + \sum_{\substack{|I|, |J| = p, q \\ I \cap J \ne \emptyset}} \sum_{a \in I \cap J} \left(W_I W_J +  \excond{W_I W_J}{\mathcal{G}_a} \right).
  \end{equation*}

  Then for $p \ne q$, using the convexity of $x \longmapsto x^2$,
  \begin{multline*}
    \var(\Gamma(F_p, F_q)) \le \frac{1}{2} \var \left[(p+q) \sum_{\substack{|I|, |J| = p, q \\ I \cap J = \emptyset}} W_I W_J \right] \\
      +  \frac{1}{2} \var \left[ \sum_{\substack{|I|, |J| = p, q \\ I \cap J \ne \emptyset}} \sum_{a \in I \cap J} \left(W_I W_J +  \excond{W_I W_J}{\mathcal{G}_a} \right)\right]
  \end{multline*}

\begin{multline*}
  \var(\Gamma(F_p, F_q))  \le  \frac{1}{2} \E \left[\left((p+q) \sum_{\substack{|I|, |J| = p, q \\ I \cap J = \emptyset}} W_I W_J \right)^2\right] \\
      +  \frac{1}{2} \var \left[\sum_{\substack{|I|, |J| = p, q \\ I \cap J \ne \emptyset}} \sum_{a \in I \cap J} \left(W_I W_J +  \excond{W_I W_J}{\mathcal{G}_a} \right)\right]
\end{multline*}
  \begin{multline*}
    2 \var(\Gamma(F_p, F_q)) \le \sum_{\substack{|I|, |J| = p, q \\ I \cap J = \emptyset}} \sum_{\substack{|K|, |L| = p, q \\ K \cap L = \emptyset}} \E [W_I W_J W_K W_L]\\
    \begin{aligned}
      &+ \E \left[\sum_{\substack{|I|, |J| = p, q \\ I \cap J \ne \emptyset}} \sum_{\substack{|K|, |L| = p, q \\ K \cap L \ne \emptyset}} \sum_{a \in I \cap J} \sum_{b \in K \cap L} W_I W_J W_K W_L \right] \\
      &+ \E \left[ \sum_{\substack{|I|, |J| = p, q \\ I \cap J \ne \emptyset}} \sum_{\substack{|K|, |L| = p, q \\ K \cap L \ne \emptyset}} \sum_{a \in I \cap J} \sum_{b \in K \cap L}W_I W_J \excond{W_K W_L}{\mathcal{G}_b} \right] \\
      &+  \E \left[ \sum_{\substack{|I|, |J| = p, q \\ I \cap J \ne \emptyset}} \sum_{\substack{|K|, |L| = p, q \\ K \cap L \ne \emptyset}} \sum_{a \in I \cap J} \sum_{b \in K \cap L} \excond{W_I W_J}{\mathcal{G}_a} W_K W_L \right] \\
      &+ \E \left[\sum_{\substack{|I|, |J| = p, q \\ I \cap J \ne \emptyset}} \sum_{\substack{|K|, |L| = p, q \\ K \cap L \ne \emptyset}} \sum_{a \in I \cap J} \sum_{b \in K \cap L} \excond{W_I W_J}{\mathcal{G}_a} \excond{W_K W_L}{\mathcal{G}_b}\right]. \\
    \end{aligned}
  \end{multline*}
We shall write
 $$ |C_{I, J, a}| = \left|\frac{\excond{W_I W_J}{\mathcal{G}_a}}{W_{I \sminus \{a\}} W_{J \sminus \{a\}}} \right| \text{ for all }I, J, a$$
 with the convention $W_{\emptyset} = 1$.

 Let us deal with each term one by one:
  \begin{itemize}
      \item If $I \cap J = \emptyset$, $K \cap L = \emptyset$, and if there is more than 2 other pairs with null intersection, the contribution of the term is 0, hence the first term is non-zero if $(I, J, K, L)$ is connected, then: $$\sum_{\substack{|I|, |J| = p, q \\ I \cap J = \emptyset}} \sum_{\substack{|K|, |L| = p, q \\ K \cap L = \emptyset}} \E [W_I W_J W_K W_L] \le \sum_{\substack{I, J, K, L \text{ connected}}} |\E [W_I W_J W_K W_L] |.$$
      \item The second term consists of the sums of product of factors indexed by connected sets since there are at least two pairs that have non-null intersection. Since $p \ne q$, $\excond{W_I W_J}{Z} = 0$ for $|I| = p$ and $|J| = q$, so if the terms are non-zero, $W_I W_J$ and $W_K W_L$ are not conditionally independent.
      \item For the third term, using self-adjointness, the terms are non-zero if $b \in I \cap J$, hence it is equivalent to:
      \begin{equation*}
         | C_{I, J, a} \E [W_{I \sminus \{b\}}  W_{J \sminus \{a\}} W_K W_L ]| = | C_{I, J, a} | |\E [W_{I \sminus \{b\}}  W_{J \sminus \{a\}} W_K W_L ]| .
      \end{equation*}
      If $b$ is the unique element that lies in the intersection, the contribution is 0, otherwise $I, J, K, L$ are connected or the contribution is $$\excond{W_I W_J}{Z} \excond{W_K W_L}{Z}= 0$$ because $|I| \ne |J|$.
      \item For the last term, it is the same argument.
  \end{itemize}
  Then, there exists a constant $C$ independent of others such that
 \begin{equation*}
     \var (\Gamma(F_p, F_q))
\le  (1 + m^2 + 2 C  m^2 + C^2 m^2 )  \sum_{\substack{I, J, K, L \text{ connected}}} |\E [W_I W_J W_K W_L] |.
 \end{equation*}

\end{proof}
% Finite chaotic decomposition
 In \cite{privault2022berry}, Privault and Serafin proves a partial fourth moment theorem for $F$ a functional of independent random variables sum of element in the first and second chaos of their own Malliavin structure. To that end, we devise another strategy which is to reexpress the remainder in the partial fourth moment theorem as a fourth order term.
% More recently, \cite{bhattacharya2020motif} derived bounds in 1-Wasserstein distance for another problem of motif estimation where a fourth moment phenomenon arises.

% \begin{corollary}
%   If $F = \sum_{p=1}^m$ where $F_p \in \mathfrak{C}_p$, assuming
%    that $$C_{I, J, a} = \frac{\excond{W_I W_J}{\mathcal{G}^a}}{W_{I \sminus \{a\}} W_{J \sminus \{a\}}}$$
%   is a finite constant independent of the set $A$ for all $I, J$,
%   then:
%   \begin{equation}
%       d_{W} (F, \Normal (0, 1)) \le C_m \sqrt{\sum_{(I, J, K, L) \textnormal{ connected}} | \E [W_I W_J W_K W_L] | } + C'_m \rho_A^2,
%   \end{equation}
%   where the constant $C_m$ and $C'_m$ grow quadratically with $m$.
% \end{corollary}
%
% \todo{Talk about constant?}
\begin{theorem}[Quantitative De Jong's theorem II]
  \label{thm-discrete-mall:full-chaos-fourth-connected}
  If $F = \sum_{p=1}^m F_p$ where $F_p \in \mathfrak{C}_p$ and let us assume:
  \begin{itemize}
      \item $F_p$ are chaos eigenfunctions \eqref{eq-discrete-mall:eigenfunction};
      \item the condition ~\eqref{eq-discrete-mall:hypothesis1};
      % \todo{Redundant to write twice the upper bound}
      % $$C_{I, J, a} = \frac{\excond{W_I W_J}{\mathcal{G}^a}}{W_{I \sminus \{a\}} W_{J \sminus \{a\}}}$$
      % is finite $\mathbb{P}$-a.s. and only depends on $Z$, independent of the set $A$ for all $I, J$.
      \item
      \begin{equation}
        \label{eq-discrete-mall:hypothesis2}
        \kappa = \sup_{I, J \subset A} \frac{\E[W_I^2] \E[W_J^2]}{\E[W_I^2 W_J^2]} < \infty
        \tag{H2}
      \end{equation}
      is independent of $A$.
  \end{itemize}
  Then:
      \begin{equation}
          d_{W} (F, \Normal (0, 1)) \le C_m \sqrt{\sum_{(I, J, K, L) \textnormal{ connected}} | \E [W_I W_J W_K W_L] | },
      \end{equation}
      where the constant $C_m$ grows quadratically with $m$, independent of all others.
\end{theorem}

\begin{proof}[Proof of theorem~\ref{thm-discrete-mall:full-chaos-fourth-connected}]
  Let us prove the upper bound of $\var \left[\Gamma(F_p, F_p)\right]$ by bounding the fourth cumulant:
  \begin{align*}
      \E [F_p^4] &=  3 \sum_{\substack{I, J, K, L \in (A, p) \\(I\cup J) \cap (K \cup L) = \emptyset}} \E [W_I W_J ] \E [W_K W_L] + \sum_{\substack{I, J, K, L \in (A, p)  \\I, J, K, L \text{ connected}}} \E [W_I W_J W_K W_L] \\
          &=  3 \sum_{I, J \in (A, p) } \E [W_I^2 ] \E[ W_J^2] - 3  \sum_{I \cap J \ne \emptyset \ne } \E [W_I^2 ] \E [W_J^2] \\
          &+ \sum_{\substack{I, J, K, L \in (A, p)  \\I, J, K, L \text{ connected}}} \E [W_I W_J W_K W_L]  \\
          &= 3 \E [F_p^2]^2 + \sum_{\substack{I, J, K, L \in (A, p)  \\I, J, K, L \text{ connected}}}\E [W_I W_J W_K W_L] - 3 \sum_{\substack{I\cap J \ne \emptyset \\ I \ne J }} \E [W_I^2] \E[W_J^2].
          % &\E [F_p^4] - 3 \E[F_p^2]^2 =  \sum_{I, J, K, L \text{ connected}} \E [W_I W_J W_K W_L] - 3 \sum_{\substack{I\cap J \ne \emptyset \\ I \ne J }} \E [W_I^2] \E[W_J^2].\\
  \end{align*}
  Then, one has:
  \begin{equation}
    |\E [F_p^4] - 3| \E[F_p^2]^2| \le \left(1 + 3 \kappa \right) \sum_{I, J, K, L \text{ connected}} |\E [W_I W_J W_K W_L]|.
  \end{equation}
\end{proof}

The assumptions may seem cumbersome, but as shown in lemma \ref{lem-discrete-mall:homogeneous-sum-eigenfunction} concerning \eqref{eq-discrete-mall:eigenfunction}, they are valid for homogeneous sums.

\section{Application to motif estimation}
\label{sec:appli}

We interest in the applications of the bounds of probability distances to asymptotic normality of subhypergraph counts in exchangeable random hypergraphs.

\subsection{Basic hypergraph definitions}
\label{sec:hypergraphs}

The hypergraph model is a generalization of graph notion that aims at model more complex model in network analysis.

\begin{definition}
  A hypergraph denoted by $G = (V, E = (e_i)_{i \in \mathcal{P}(V)})$ on a finite set $V = V(G)$ is a family of subsets of $V$ called hyperedges. Vertices in a hypergraph are adjacent if there is a hyperedge which contains them. The vertices not in any edge are the isolated vertices of $G$. A hypergraph is connected if it contains no isolated vertices and if the intersection graph of $E$ is connected.
\end{definition}
We denote by $[e]$ the set of vertices of the hyperedge $e$.
\begin{definition}
  A $k$-uniform hypergraph $G = (V, E)$ is a hypergraph where each hyperedge has cardinality $k$. In particular, such hypergraph has hyperedge set in $\binom{V}{k}$, the collection of $k$-tuples of the set of vertices $V$.
\end{definition}

\begin{definition}
  For $k > 3$, a subhypergraph (or simply subgraph) of a hypergraph $G = (V, E)$ is a hypergraph $H = (V', E')$ such that $V' \subset V$ and $E' \subset E \cap \binom{V}{k}$.
\end{definition}
We denote by $v_H$ and $e_H$ the number of vertices and number of hyperedges of a hypergraph $H$ respectively.

A 2-uniform hypergraph is a graph. A 3-uniform hypergraph is a hypergraph whose hyperedges are triangles only.
We also denote by $G^{(j)}$ the hypergraph induced by the hypedges of cardinality $j \le k$ included in the hyperedges of the $k$-uniform hypergraph $G$.

\subsection{Exchangeable random hypergraphs}
\label{sec:exchangeable-hypergraphs}
% The generations of random hypergraphs extends those of random graphs.
% On the contrary to a simplicial 2-complex, that structure does not include the edges between two vertices of a common triangle.
The random hypergraphs are natural extensions of random graphs. A vast majority of the literature deals with the Erdös-Rényi model and its generalization. It is an example of exchangeable random hypergraphs.
% The family of exchangeable random graphs is central in the theory of dense graphs.
\begin{definition}
  A $k$-uniform exchangeable random hypergraph $\mathbf{G}$ of vertex set $V = [n]$ is defined by the set of $\{0,1\}$-valued random variables $(X_\alpha, \alpha \subset \binom{[n]}{k})$ such that:
  \begin{itemize}
    \item one associates each realization of the random variables a hypergraph $([n], E)$ with $\alpha \in E$ if and only if $X_\alpha = 1$;
    \item $(X_\alpha)$ form an exchangeable array, i.e. $X_{(\sigma (u))_{u \in \alpha}} \overset{d}{=} X_\alpha$.
  \end{itemize}
\end{definition}
One can formulate a recipe for exchangeable random hypergraphs as done in \cite[definition 2.8]{austin_exchangeable_2008}.
% An exchangeable random hypergraph corresponds to an exchangeable measure $\mu$ on $\{0,1\}^{\binom{V}{k}}$.
Fix a sequence of ingredients which consist of a sequence of sample spaces and probability kernels that determine the presence of $k$-hyperedges in the hypergraph based on the indicators $X_\beta$ for  $(k-1)$-hyperedges $\beta$:
\begin{equation*}
  (\{*\}), (V, P_1), (\{0,1\}, P_2), (\{0,1\}, P_3),\ldots, (\{0,1\}, P_{k-1}), (\{0,1\}, P_k)
\end{equation*}
where we write $\{*\}$ for a one-point space, $(P_k)_{k \in \N}$ is a family of probability kernels such that for all $k \in \N$, $P_k$ is a probability kernel from $\prod_{j=0}^{k-1}\{0, 1\}^{\binom{V}{j}}$ to $\{0, 1\}^{\binom{V}{k}}$.
% Fix $\mu_0$ a measure on $\mathcal{Z}_0$.
\begin{itemize}
    % \item Choose $z_{\emptyset} \in \mathcal{Z}_0$ at random according to some law $\mu_0$;
    \item Color each vertex $s \in V$ by some $x_s \in \{0,1\}$ chosen independently according to $P_1(*, \cdot)$;
    \item Color each edge $a = \{ s,t\} \in \binom{V}{2}$ by some $x_a \in \{
    0,1\}$ chosen independently according to $P_2(*, x_s, x_t, \cdot)$;
    \begin{center}
        $\vdots$
    \end{center}
    \item Color each $(k-1)$-hyperedge $u \in \binom{V}{k-1}$ by some $x_u \in \{0,1\}$ chosen independently according to $P_{k-1}(*, (x_s)_{s \in \binom{[u]}{1}}, *,  \ldots, *,  (x_v)_{v \in \binom{[u]}{k-2}} , \cdot)$;
    \item  Color each $k$-hyperedge $e \in \binom{V}{k}$ by some color $x_e \in \{0,1\}$ chosen independently according to $P_{k}(*, (x_s)_{s \in \binom{[e]}{2}}, *,  \ldots, *,  (x_u)_{u \in \binom{[e]}{k-1}} , \cdot)$.
\end{itemize}
% It consists in sampling first edges, then 3-hyperedges, and so on up to the rank $k$.
% \begin{remark}
%   We have relative independence, as the choices of hyperedges are independent relatively to the previous stages.
%   % The recipe proceeds by generating independent hyperedges relatively to the previous built level.
% \end{remark}
%
	% Fix $n \in \N$, the Erdös-Rényi model $\mathbb{G}(n,p)$ corresponds to a $W$-random graph with $W = p$. The vertex set $V$ is $[n]$.
  \begin{example}[Erdös-Rényi random model]
    The randomness intervenes at the level of edges. $P_1(*, \cdot)$ is the uniform distribution on $V$.  We color each edge $a = \{s,t\} \in \binom{V}{2}$ by some $z_a \in \{0, 1\} $ chosen independently according to $P_2(*, x_s, x_t, \cdot) \overset{d}{=}\mathcal{B}(p)$ the Bernoulli distribution with parameter $p$ for some $p \in [0,1]$ which is called the \textit{edge density}.
  \end{example}

  \begin{example}[Stochastic block model]
        A stochastic block model corresponds to a model where there are communities, and each edge has a probability of belonging to the model according to the community of the vertices that the edge links. Likewise, the randomness intervenes at the level of the edges. Let a partition $V = C_1 \sqcup \ldots \sqcup C_q $. Let $(p_{i,j})_{i, j \in \llbracket 1, q \rrbracket^2}$ a sequence of reals in $[0, 1]$. We can assign a community to each vertex $s$, let call it $c(s)$. Then:
        \begin{itemize}
          \item $P_1(*, \cdot)$ is the uniform distribution;
          \item $P_2(*, z_{s}, z_{t}, \cdot)  \overset{d}{=} \mathcal{B}(p_{c(s), c(t)})$.
      \end{itemize}
  \end{example}
The natural extension of the Erdös-Rényi model denoted $\mathbb{G}^{(3)}(n, p_n)$ consists of having $$P_3(*, x_{st}, x_{tu}, x_{us}) \overset{d}{=}\mathcal{B}(p_n),$$  i.e. we draw every triangle of the hypergraph with probability $p_n$.
We also consider another random model based on the recipe.
% Let us consider the random hypergraph model that is presented for instance in \cite{elek2012measure} and detailed for the particular case of 3-uniform hypergraphon.
Let $(\mathbb{T}^{(3)}(n, q_n, p_n))_{n \in \N}$ the sequence of $3$-uniform hypergraphs such that for $(s,t,u) \in V^3$:
    \begin{itemize}
      \item   $$P_{2}(*, x_s, x_t) \overset{d}{=}\mathcal{B}(q_n);$$
      \item $$P_{3}(*, x_{st}, x_{tu}, x_{us}) \overset{d}{=}\mathcal{B}(p_n).$$
    \end{itemize}
It differs from $\mathbb{G}^{(3)}(n, p_n)$ in many ways as pointed out by \cite[Example 23.11]{lovasz2012large}, but we note that $\mathbb{G}^{(3)}(n, p_n)$ and $\mathbb{T}^{(3)}(n, 1, p_n)$ have the same law.
%We will consider that general model $\mathbb{T}^{(3)}(n, q_n, p_n)$
The case $q_n < 1$ has not been much studied in the literature.
The functional identities in Section~\ref{sec:funct-ident} can be applied to random hypergraphs in the same way as for random graphs \cite[corollary 2.27]{janson2000random}.
In that section, we consider once for all $A$ to be the set of hyperedges. We use the notation $A$ for other purposes.

\subsection{Motif estimation in random hypergraphs}
\label{sec:motif-estimation}

% Coulson
One of the oldest problem of motif estimation is subgraph counting in random graphs. Small subgraph counts can be used as summary statistics for large random graphs.
The asymptotic normality of subgraph count in Erdös-Rényi model is well-known, as well as the convergence rate \cite{janson1991asymptotic}.
% , it is to say that we know how dense we need the random graph as to ensure the asymptotic normality of the estimator.
There are many extensions that revolve around the definition of a random graph as a sequence of independent random variables, for example a clique complex of Bernoulli random graphs. In this work, we study subgraph counting in 3-uniform random hypergraphs.
To the best of our knowledge, this is the first paper about asymptotic normality of subgraph counting of such models.

The number of subhypergraphs of $\mathbb{G}^{(3)}(n, p_n)$ isomorphic to $G$ is
\begin{equation}
    M_G =  \sum_{\substack{H \in \binom{[n]}{3} \\H \simeq G}} \prod_{\alpha \in H} \hat X_\alpha.
\end{equation}
% We denote by $\text{Aut}(G)$ the automorphism group of $G$, that is the set of permutation on vertices.
For $\sigma \in \text{Aut}(G)$, $(x, y, z) \in E(G)$ if and only if $(\sigma(x), \sigma(y), \sigma(z)) \in E(G)$.
% The random variable $N_G$ has expectation $n^{|V(G)|} p_n^{e_G} / |\text{Aut}(G)|$ and a nice Hoeffding decomposition \cite[p.11(115)]{de1996central}.
The random variable $M_G$ has a finite Hoeffding decomposition \cite[p.11(115)]{de1996central}.
Since $\hat X_\alpha = p_n + (\hat X_\alpha - p_n)$, $M_G$ admits the decomposition:
\begin{equation}
    M_G = \sum_{\substack{H \in \binom{[n]}{3} \\H \simeq G}} \sum_{J \subseteq H} p_n^{|H| - |J|} \prod_{\alpha \in J} (\hat X_\alpha - p_n),
\end{equation}
where the summation extends over all subsets $J$ of $I$, in virtue of the inclusion-exclusion principle. By interchanging the sums, we find the chaotic decomposition of $M_G - \E [M_G]$ that is:
\begin{align*}
    M_G - \E [M_G]  &=\sum_{\substack{H \in \binom{[n]}{3} \\H \simeq G}} \sum_{\substack{J \subseteq H \\ J \ne \emptyset}} p_n^{|I| - |J|} \prod_{\alpha \in J} (\hat X_\alpha - p_n), \\
    &=  \sum_{\substack{H \in \binom{[n]}{3} \\H \simeq G}} \sum_{j=1}^{e_G}  p_n^{e_G - j} \sum_{\substack{J \subset H \\ |J| = j} } \prod_{\alpha \in J} (\hat X_\alpha - p_n) \\
    &= \sum_{j=1}^{e_G}  p_n^{e_G - j} \sum_{ |J| = j } \prod_{\alpha \in J} (\hat X_\alpha - p_n) \left( \sum_{\substack{H \in \binom{[n]}{3} \\H \simeq G, H \supseteq J}} 1 \right) \\
    &= \sum_{j=1}^{e_G} \pi_j (M_G),
\end{align*}
where
\begin{equation}
    \pi_k (M_G) = p_n^{e_G - j} \sum_{ |J| = j }\left( \sum_{\substack{H \in \binom{[n]}{3} \\H \simeq G, H \supseteq J}} 1 \right) \prod_{\alpha \in J} \hat Y_\alpha
\end{equation}
with $\hat Y_\alpha$ is the centered version of $\hat X_\alpha$ for all $\alpha$ hyperedges of $K_n$.
% So, denoting $|E(G)| = e_H$, for $(i_1, \ldots, i_k) \in [n]^{k}$,
% \begin{equation*}
%     \tilde{f}_k (i_1, \ldots, i_k) = p_n^{|E(G)| - k} \sum_{(j_1, \ldots j_{e_H - k}) \in [n]^{e_H - k}} \indicator_{E_H} (i_1, \ldots, i_k, j_1, \ldots j_{e_H - k}).
% \end{equation*}
% \todo{remove kernel notations}
We note that the decomposition above corresponds to the Hoeffding decomposition of the U-statistics with
\begin{equation}
  \label{eq-discrete-mall:motif-count-indep-hoeffding-term}
  W_J \propto \left( \sum_{\substack{H \in \binom{[n]}{3} \\H \simeq G, H \supseteq J}} 1 \right) \prod_{\alpha \in J} \hat Y_\alpha.
\end{equation}
We proceed in the same manner in $\mathbb{T}^{(3)}(n, q_n, p_n)$. Define $N_G$ the number of subhypergraphs isomorphic to $G$

\begin{equation}
  N_G = \sum_{\substack{H \in \binom{[n]}{3} \\H \simeq G}} \prod_{\alpha \in H} X_\alpha.
\end{equation}
Here, $(X_\alpha)_{\alpha \in \binom{[n]}{3}}$ is a sequence of conditionally independent Bernoulli random variables given $Z = \mathbb{G}(n, q_n)$. The chaos decomposition yields:
\begin{equation}
  \label{eq-discrete-mall:N_G-decomposition}
  \begin{aligned}
    N_G &= \sum_{\substack{H \in \binom{[n]}{3} \\H \simeq G}} \sum_{J \subseteq H} \prod_{\beta \in H \sminus J} \excond{X_\beta}{\mathbb{G}(n, q_n)}\prod_{\alpha \in J} (X_\alpha - \excond{X_\alpha}{\mathbb{G}(n, q_n)}) \\
    &= \sum_{\substack{H \in \binom{[n]}{3} \\H \simeq G}} \sum_{J \subseteq H} p_n^{|H|-|J|} \ind{(H \sminus J)^{(2)} \subset \mathbb{G}(n, q_n)} \prod_{\alpha \in J} (X_\alpha - \excond{X_\alpha}{\mathbb{G}(n, q_n)}). \\
\end{aligned}
\end{equation}

Hence, $N_G - \excond{N_G}{\mathbb{G}(n, q_n)}$ reads off:
\begin{equation}
  \sum_{\substack{H \in \binom{[n]}{3} \\H \simeq G}} \sum_{\emptyset \ne J \subseteq I} p_n^{|H|-|J|} \ind{(H \sminus J)^{(2)} \subset \mathbb{G}(n, q_n)} \prod_{\alpha \in J} (X_\alpha - \excond{X_\alpha}{\mathbb{G}(n, q_n)}).
\end{equation}
% \begin{equation}
%   M_G - \excond{M_G}{\mathbb{G}(n, q_n)} = \sum_{j=1}^{|E(G)|} \excond{\hat X_\alpha}{\mathbb{G}(n, q_n)}^{|E(G)|}  \sum_{ |J| = j } \prod_{\alpha \in J} (\hat X_\alpha - \excond{X_\alpha}{\mathbb{G}(n, q_n)}) \left( \sum_{\substack{I \in \binom{[n]}{r} \\I \simeq G, I \supseteq J}} 1 \right)
% \end{equation}
The corresponding degenerate U-statistics in the decomposition are given for $J \subset \binom{[n]}{3}$ by
\begin{equation}
  \label{eq-discrete-mall:motif-count-condindep-hoeffding-term}
  W_J \propto \left( \sum_{\substack{I \in \binom{[n]}{3} \\I \simeq G, I \supseteq J}} 1 \right) \prod_{\alpha \in J}Y_\alpha,
\end{equation}
where $Y_\alpha$ is the centered version of $X_\alpha$ given $\mathbb{G}(n, q_n)$ and:
and:
\begin{equation*}
  w_{J} = \left( \sum_{\substack{I \in \binom{[n]}{3} \\H \simeq G, I \supseteq J}}  p_n^{|H|-|J|} \ind{(H \sminus J)^{(2)} \subset \mathbb{G}(n, q_n)} \right).
\end{equation*}
% \todo{Not sure if we condition by the whole graph?}
%
% \todo[]{Coulson, motif estimate}
Historically, normal approximation for subgraph counting had been dealt with the method of moments \cite{rucinski1988small} which requires tedious computations, but is quite adapted to this application. In particular, thresholds of asymptotic normality for the density of edges are obtained in function of $n$ the number of vertices.
% \chris{That result dates back to Erdös and Rényi in the case where $G$ is a balanced graph.} There is an equivalent result for $k$-uniform hypergraphs.
% Let $d(H)$ the density of a hypergraph $G$, i.e. $d(H) = \frac{e_H}{v_H}$.
% \todo{balancedness of hypergraph}
% \begin{definition}[Balanced graph]
%   A hypergraph $G$ is balanced if the density of $G$ is equal to the maximal subgraph density i.e. $d(H) = \max \{d(K) \: : \: K \subseteq H\}$.
% \end{definition}
% The threshold for any graph $G$ was first found by Bollobas \cite{bollobas1981threshold}.
In \cite{barbour1989central}, the authors used Stein's method to derive convergence rates of the number of subgraph counting in random graphs in the 1-Wasserstein distance.
The combination with Malliavin calculus has brought another feature to the usual coupling constructions in Stein's method, leveraging chaos representation property for independent identically distributed (see \cite{privault2008stochastic}).
$M_G$ is a Rademacher functional, so it has its Walsh chaotic decomposition. It has led to applications to subgraph counting in random graphs \cite{privault2020normal} and percolation problems \cite{krokowski2017discrete}.
By applying theorem \ref{thm-discrete-mall:full-chaos-fourth-connected} to \eqref{eq-discrete-mall:motif-count-indep-hoeffding-term}, we obtain a quantitative version of the main theorem in \cite{de1996central} as well as its counterpart for $\mathbb{T}^{(3)}(n, q_n, p_n)$.
To the best of our knowledge, there is no study of normal approximation of motif estimation in $\mathbb{T}^{(3)} (n, q_n, p_n)$.
Let us denote $\bar M_G$ and $\bar N_G$ the respective rescaled statistic of the number of isomorphic copies of $G$ with respect to their expectation, and let $\tilde N_G$ the rescaled statistic with respect to its conditional mean.

\begin{theorem}
  \label{thm-motif_estimation:dejong-quant-hypergraph}
Let $G$ a hypergraph without isolated vertices. Then,
\begin{equation}
  d_{W} (\tilde N_G, \Normal (0, 1)) \le C_{e_G} \sqrt{\sum_{(I, J, K, L) \textnormal{ connected}} | \E [W_I W_J W_K W_L] |  / \var [\tilde N_G]}.
\end{equation}
\end{theorem}

\begin{proof}[Proof of Theorem \ref{thm-motif_estimation:dejong-quant-hypergraph}]
  We check whether the conditions of theorem~\ref{thm-discrete-mall:full-chaos-fourth-connected} hold.
  For both statistics, the \eqref{eq-discrete-mall:eigenfunction} assumption holds.
  By conditional independence of $(X_\alpha)_{\alpha \in  \binom{[n]}{3}}$, we have:
  \begin{align*}
    \frac{\E[W_I^2] \E[W_J^2]}{\E [W_I^2 W_J^2]} &\propto \frac{\E \left[\prod_{\alpha \in I} \E [Y_\alpha^2 |Z] \prod_{\alpha \in J} \E [Y_\alpha^2 | Z]\right]}{\prod_{I \sminus J} \E [Y_\alpha^2] \prod_{J \sminus I} \E [Y_\alpha^2] \prod_{I \cap J} \E [Y_\alpha^2]} \\
    % &= \frac{\left(p_n(1 - p_n)\right)^{|I| + |J|} q_n^{|I^{(2)}| + |J^{(2)}|}}{\left(p_n(1 - p_n)\right)^{|I \sminus J| + |J \sminus I| + |I \cap J|}}  \\
    &= \left(p_n (1 - p_n)\right)^{|I| + |J| - |I \cup J|} q_n^{|I^{(2)}| + |J^{(2)}| - |I^{(2)} \cup J^{(2)}|}\le 1.\\
  \end{align*}
  Let us note that for all $a$, $W_{I \sminus \{a\}}$ is non-zero with the definition of $N_G - \excond{N_G}{Z}$. Let $W_I = w_I \prod_{i \in I} X_i$, then for $a \in I \cap J$:
    \begin{align*}
        \excond{W_I W_J}{\mathcal{G}_a} &= w_I w_J \prod_{i \in I \sminus \{a\}} Y_i \prod_{j \in J \sminus \{a\}} Y_j \excond{Y_a^2}{Z} \\
        &=  \frac{w_I w_J}{w_{I \sminus \{a\}} w_{J \sminus \{a \}} } \excond{Y_a^2}{Z} W_{I \sminus \{a\}} W_{J \sminus \{a \}} \\
        &= C_{I, J, a} W_{I \sminus \{a\}} W_{J \sminus \{a \}},
    \end{align*}
    with
    % $C_{I, J, a} = \frac{w_I w_J}{w_{I \sminus \{a\}} w_{J \sminus \{a \}} }  \excond{Y_a^2}{Z} < +\infty$ $\mathbb{P}$-almost surely.
    \begin{equation*}
      C_{I, J, a} = \frac{w_I w_J}{w_{I \sminus \{a\}} w_{J \sminus \{a \}} }  \excond{Y_a^2}{Z} < +\infty\quad \mathbb{P}\text{-almost surely}.
    \end{equation*}
%     The same goes for $W_I W_J \propto \prod_{\alpha \in I} Y_{\alpha} \prod_{\beta \in J} \hat Y_J$ and $W_I W_J \propto \prod_{\alpha \in I} \hat Y_{\alpha} \prod_{\beta \in J} \hat Y_J$.
% Then,
\end{proof}

We deduce those convergence rates for $p_n < c < 1$ for some $c$.
\begin{theorem}
    \label{thm-discrete-mall:hypergraph-condindep-threshold}
    Let $G$ a hypergraph without isolated vertices. Then, we have
    % for $p_n \xrightarrow{n \to + \infty} 0$ and $q_n \xrightarrow{n \to + \infty} 0$:
    \begin{equation}
      d_{W}(\tilde N_G, \Normal(0, 1)) \lesssim \left( \min_{\substack{H \subset G \\ e_H > 1}} \{n^{v_H} p_n^{e_H} \} \right)^{-1/2}
  \end{equation}
  and
    \begin{equation}
      \label{eq-discrete-mall:hypergraph-condindep-threshold}
        d_{W}  (\tilde N_G, \Normal(0, 1))  \lesssim
        % (1 - p_n q_n)
        \left( \min_{\substack{H \subset G \\ e_H > 1}} \{n^{v_H} p_n^{e_H} q_n^{e^{(2)}_H}\} \right)^{-1/2},
    \end{equation}
    where $e^{(2)}_H$ is the number of edges included in the hyperedges of $H$.
  \end{theorem}

\begin{proof}[Proof of theorem~\ref{thm-discrete-mall:hypergraph-condindep-threshold}]
  % We proceed as in the previous proof given the decomposition of $M_G$ and borrow the previous notations.
  We are left to upper bound the quantity:
    \begin{multline*}
        \sum_{(I, J, K, L) \textnormal{ connected}} |\E [W_I W_J W_K W_L]| \\
         \propto_Z \sum_{(I, J, K, L) \textnormal{ connected}}  \left| \E \left[ \excond{\prod_{\alpha \in I} Y_\alpha \prod_{\alpha \in J} Y_\alpha \prod_{\alpha \in K} Y_\alpha \prod_{\alpha \in L} Y_\alpha}{Z}   \right] \right|
    \end{multline*}
    where the notation $\propto_Z$ accounts for an equality up to a factor depending only on $Z$.
    The terms are non-zero if and only if $\alpha$ lies in at least two elements of the quadruple, i.e. if $\alpha$ does not lie in $I \setminus (J \cup K \cup L)$, etc. Then, the number of non-zero terms is $I \cup J \cup K \cup L$.
    We recall that:
    \begin{align*}
      \E [Y_\alpha|Z] &= 0 \\
      \E [Y_\alpha^2|Z] &= p_n (1 - p_n) \ind{\alpha^{(1)} \in Z} \prod_{i=1}^3\ind{\alpha^{(i)} \in Z} \\
      \E [Y_\alpha^3 |Z] &= p_n (1-p_n) (1-2p_n)\prod_{i=1}^3\ind{\alpha^{(i)} \in Z} \lesssim p_n (1-p_n)\prod_{i=1}^3\ind{\alpha^{(i)} \in Z}\\
      \E [Y_\alpha^4 |Z ] &= p_n (1-p_n)(1-3p_n(1-p_n)) \prod_{i=1}^3\ind{\alpha^{(i)} \in Z}\lesssim p_n(1-p_n) \prod_{i=1}^3\ind{\alpha^{(i)} \in Z}.
    \end{align*}
    % Thus,
    % \begin{multline*}
    %   \sum_{(I, J, K, L) \textnormal{ connected}} |\E [W_I W_J W_K W_L]| \\
    %   \lesssim \sum_{(I, J, K, L) \textnormal{ connected}} \left(p_n (1-p_n)\right)^{|I \cup J \cup K \cup L|} q_n^{|I^{(2)} \cup J^{(2)} \cup K^{(2)} \cup L^{(2)}|}.
    % \end{multline*}
    Now, we remark $I, J, K, L$ are respectively isomorphic to $A, B, C, D$ subhypergraphs of $G$.
    % We note that there are several subsets $I$ isomorphic to $A$ in the random hypergraph, etc.
    Hence, we can sum first over $(A, B, C, D)$, and then over all the quadruples $(I, J, K, L)$ whose components are respectively isomorphic to the ones of the fixed quadruple $(A, B, C, D)$.
    We shall write:
    \begin{equation*}
        \sum_{I, J, K, L} \cdot = \sum_{A, B, C, D \subset G} \sum_{\substack{I \simeq A, J, \simeq B \\ K \simeq C, L \simeq D}} \cdot := \sum_{A, B, C, D} \sum_{I, J, K, L}^{*A, B, C, D} \cdot
    \end{equation*}
    $v(A)$ denotes the number of vertices in $A$.
    % Let us compute $|\{I \in \binom{[n]}{r}:\: I \simeq A\}|$. It is the number of subgraphs which are isomorphic to a sub-collection $A$ of $G$. We know that it is bounded by the number of collection of vertices whose cardinal is $v(A)$. More generally,
    We have that $|\{I, J, K, L \in \binom{[n]}{r}:\: I \simeq A, J \simeq B, K \simeq C, L \simeq D\}|$ is bounded by the number of collection of vertices of cardinal $v(A \cup B \cup C \cup D)$. By a counting argument,
    % we see that $|\{I \in \binom{[n]}{r}:\: I \simeq A\}|$ is of order of magnitude $n^{v(A)}$.
    we see that is of order $n^{v(A \cup B \cup C \cup D)}$. Because $(I, J, K, L)$ is connected and copies of subhypergraphs of $G$, we also have that $|I \cup J \cup K \cup L| = |A \cup B \cup C \cup D|$ and $|I^{(2)} \cup J^{(2)} \cup K^{(2)} \cup L^{(2)}| = |A^{(2)} \cup B^{(2)} \cup C^{(2)} \cup D^{(2)}|$. Hence, for a fixed connected quadruple $(I, J, K, L)$ associated to $(A, B, C, D)$,
    \begin{multline*}
        |\E [W_I W_J W_K W_L]| \\
        \lesssim  w_I w_J w_K w_L \, n^{v(A \cup B \cup C \cup D)} p_n^{|A \cup B \cup C \cup D|} q_n^{|A^{(2)} \cup B^{(2)} \cup C^{(2)} \cup D^{(2)}|}.\\
    \end{multline*}
    Let us bound the variance of $N_G - \excond{N_G}{\mathbb{G}(n, q_n)}$:
        \begin{align*}
          \var^2[N_G - \excond{N_G}{\mathbb{G}(n, q_n)}] &= \left(\sum_{I \cap J \ne \emptyset}\E [W_I W_J] \right)^2 = \sum_{I \cap J \ne \emptyset}\left( \E [W_I^2] + \E [W_J^2] \right)^2 \\
        &=  \frac{1}{2^2}  \sum_{\substack{A , B \subset G \\ A \cap B \neq \emptyset}} \left(\sum_{I}^{*A}\E [W_I^2] + \sum_{J}^{*B}\E [W_J^2] \right)^2. \\
        \end{align*}
        For a fixed connected quadruple $(A, B, C, D)$, by applying repeatedly the inequality $a^2 + b^2 \ge 2 ab$, we get:
        \begin{multline*}
          \var^2[N_G - \excond{N_G}{\mathbb{G}(n, q_n)}] \\
          \begin{aligned}
        &\ge \frac{1}{16} \left(\sum_{I}^{*A}\E [W_I^2] + \sum_{J }^{*B}\E [W_J^2] + \sum_{K }^{*C}\E [W_K^2] +  \sum_{L}^{*D}\E [W_L^2]  \right)^2\\
        &\ge \frac{1}{16} \left(\sum_{I}^{*A}\E [W_I^2] \times \sum_{J }^{*B}\E [W_J^2] \times \sum_{K }^{*C}\E [W_K^2] \times  \sum_{L}^{*D}\E [W_L^2]\right)^{1/2}.
        \end{aligned}
    \end{multline*}
  Then using that $\E[W_I^2] = w_I^2 q_n^{|I^{(2)}|}(1 - p_n)^{|I|}p_n^{|I|} = q_n^{|A^{(2)}|} (1 - p_n)^{|A|}p_n^{|A|}$, so
  \begin{equation*}
    \sum_{I}^{*A}\E [W_I^2] =  \sum_{I}^{*A} w_I^2 (1 - p_n)^{|A|}p_n^{|A|} q_n^{|A^{(2)}|}= n^{v(A)} (1 - p_n)^{|A|}p_n^{|A|} q_n^{|A^{(2)}|}.
  \end{equation*}

  In particular, one has for a fixed quadruple $(I, J, K, L)$ and associated $(A, B, C, D)$:
  \begin{multline}
    \label{eq-discrete-mall:var-square}
    \var^2[N_G - \excond{N_G}{\mathbb{G}(n, q_n)}] \geq  \\
    \frac{1}{16} w_I w_J w_K w_L  \left(n^{v^*(A, B, C, D)} (p_n(1 - p_n))^{e^*(A, B, C, D)}q_n^{e^{(2)*}(A, B, C, D)}\right)^{1/2}
  \end{multline}
  where $v^*(A, B, C, D) = v(A) + v(B) + v(C)+v(D)$ and $e^*(A, B, C, D) = |A| + |B| + |C| + |D|$ and $e^{(2)*}(A, B, C, D) = |A^{(2)}| + |B^{(2)}| + |C^{(2)}|+|D^{(2)}|$.
  % \begin{multline*}
  %   \var^2[N_G - \excond{N_G}{\mathbb{G}(n, q_n)}] \propto \\  \sum_{A, B, C, D}  \left(n^{v(A) + v(B) + v(C)+v(D)} (p_n(1 - p_n))^{|A| + |B| + |C| + |D|} q_n^{|A^{(2)} \cup B^{(2)} \cup C^{(2)} \cup D^{(2)}|}
  %   \right)^{1/2}.
  % \end{multline*}
  %
  It yields the result.
  %
  % \begin{equation*}
  %   d_{W}(Z(H, \mathbb{G}^{(3)}(n, p_n)), \Normal(0, 1))
  %    \lesssim \sum_{A, B, C, D} n^{v^*(A, B, C, D) } p_n^{e^*(A, B, C, D) } q_n^{e^{(2)*}(A, B, C, D) }
  % \end{equation*}
  %
  Using the Lemma~9 of \cite{de1996central}, for any quadruple $(I, J, K, L)$ of collections $I, J, K, L$ isomorphic to (sub)collections of $H$, there exists two (not both empty) subcollections of $G$, say $M$ and $M'$, which may contain a nonzero number of isolated vertices, say $i_M$ and $i_{M'}$, such that
  \begin{equation}
    v(M) + v(M') + i_M + i_{M'} = v(I) + v(J) + v(K) + v(L) - 2 v(I \cup J \cup K \cup L),
  \end{equation}

  \begin{equation}
    |M^{(2)}| + |M'^{(2)}| = |I^{(2)}| + |J^{(2)}| + |K^{(2)}| + |L^{(2)}| - 2|I^{(2)} \cup J^{(2)} \cup K^{(2)} \cup L^{(2)}|
  \end{equation}
  and by extension:
  \begin{equation}
    |M| + |M'| = |I| + |J| + |K| + |L| - 2|I \cup J \cup K \cup L|.
  \end{equation}
  As $G$ does not have isolated vertices, so do $M$ and $M'$.
  As $M$ and $M'$ are subcollections of $G$, their average degree does not exceed $m(G)$. Hence, by theorem~\ref{thm-motif_estimation:dejong-quant-hypergraph},
  \begin{equation*}
    d_{W} (\tilde N_G, \Normal (0, 1)) \lesssim (n^{v(M) + v(M')} p_n^{|M| + |M'|} q_n^{|M^{(2)}| + |M'^{(2)}| })^{1/2}.
  \end{equation*}
   Thus, \eqref{eq-discrete-mall:hypergraph-condindep-threshold} follows.
  % By considering the multiplicative constant in $(1 - p_n)$.
  % The power is:
  % \begin{multline*}
  %   |A \cup B \cup C \cup D| - |A| + |B| + |C| + |D| \\
  %   \begin{aligned}
  %      &= - |A \cap B| - |A \cap C| - |A \cup D| - |B \cap C| - |B \cap D| - |C \cap D| + |A \cap B \cap C| \\
  %      &+ |A \cap C \cup D| + |B \cap C \cap D| \\
  %      &\ge -4 e_G
  %   \end{aligned}
  % \end{multline*}
  The first result for $M_G$ is obtained with $q_n = 1$.
\end{proof}

\begin{remark}
  This bound is relevant only for the regime $p_n \xrightarrow{n \to \infty} 0$.
\end{remark}
% \todo[]{Useful or not}
% \begin{remark}
%   As in \cite{privault2020normal}, one can consider random weights to the components of subgraph counts as long as they are $\sigma(Z)$-measurable.
% \end{remark}
%
The Malliavin structure for conditionally independent random variables yields a chaos decomposition and rates of normal convergence of the conditionally centered statistic given $Z$.

\subsection{A modified Hoeffding decomposition}
In that section,
% we consider the more general problem tackled in chapter~\ref{chap:mall_calc_condindep}.
we readopt the notations of the previous chapter by denoting $A$ the index set of the random variables. Let another set $\hat A$ that index auxiliary random variables in addition to $(\hat X_\beta)_{\beta \in \hat A}$.
We shall write the sequence of conditionally independent random variables given $Z$,
$\mathsf{X} = (X_{\alpha}, \ldots, \hat X_\beta, \ldots)_{\alpha \in A, \beta \in \hat A}$ where the subsequence $(\hat X_{\beta})_{\beta \in \hat A}$ is a sequence of independent random variables, and $\sigma(Z) = \sigma(\hat X_a, a \in A)$. This setting is new to the best of our knowledge, and is specifically tailored for the application in $\mathbb{T}^{(3)}(n, q_n, p_n)$. We assume that $A$ is the set of $3$-hyperedges, $\hat A$ is the set of edges included in the hyperedges of $A$ and
\begin{equation}
  \label{eq-discrete-mall:cond_var_cascade}
  X_\alpha = g(U_\alpha) \prod_{b \subset \alpha} \hat X_b
\end{equation}
where $(U_\alpha)_{\alpha \in A}$ forms a sequence of conditionally independent random variables given $\hat X$, following the uniform distribution.
\begin{lemma}
  The sequence $\mathsf{X}$ is a sequence of conditionally independent random variables.
\end{lemma}

\begin{proof}
   Since, by assumption, for $f$ bounded and $(\alpha, \beta) \in \hat A^2$ such that $\alpha \neq \beta$:
  \begin{equation*}
    \excond{f(\hat X_\beta)}{\hat X_\alpha, Z} = \excond{f(\hat X_\beta)}{Z},
  \end{equation*}
  the subsequence $(\hat X_\beta)_{\beta \in \hat A}$ is a sequence of conditionally independent random variables given $Z$.
  Let $\alpha \in A$ and $\beta \in \hat A$, by definition
  \begin{equation*}
    \excond{f(X_\alpha)}{\hat X_\beta, Z} = \excond{f(X_\alpha)}{Z},
  \end{equation*}
  and:
  \begin{equation*}
    \excond{f(\hat X_\beta)}{X_\alpha, Z} = \excond{f(\hat X_\beta)}{Z}
  \end{equation*}
  because $\hat X_\beta$ is a function of $Z$.
\end{proof}

\begin{remark}
  That type of sequence is a degenerate case of sequence of conditionally independent random variables since the $\hat X_\alpha$ are constant given $Z$.
\end{remark}

% \chris{The key is that in our application there exists a function $g$, such that $\mathbb{P}^{X_a | Z} = \delta_{g(Z)}$.}
%
For our purpose, the following lemma shows the commutation relation.
% on $\dom{D}$ in that Malliavin framework when $X$ consists of Bernoulli random variables.

\begin{lemma}
  \label{lem-discrete-mall:commute-cond-exp}
  For $F \in L^2(E_A)$ a homogeneous sum of conditionally independent random variables $\mathsf{X}$ and $\alpha \in A$ and $\beta \in \hat A$ such that
 \begin{equation}
  \label{eq-discrete-mall:cond-commut-hatX-X}
    \excond{\excond{F}{\hat X^{\{\beta\}}, X}}{X^{\{\alpha\}}, \hat X} = \excond{\excond{F}{X^{\{\alpha\}}, \hat X}}{X, \hat X^{\{\beta\}}}.
  \end{equation}
\end{lemma}

\begin{proof}[Proof of lemma~\ref{lem-discrete-mall:commute-cond-exp}]
  % \begin{itemize}
    % \item If $\alpha, \beta \in A^1$, then the conditional expectation are taken with respect to at least $\sigma(Z) = \sigma(\hat X_a, a \in A^0)$. It  boils down to lemma~\ref{lem-discrete-mall:gradient-commut}.
    % \item
    % Let $(U_\beta)_{\beta \in A \cup \hat A}$ a sequence of Bernoulli random variables of parameter $p_n$.
    % Since $X_\beta = \frac{\ind{U_\beta \le p_n}}{p_n} \hat X_\beta$, we have:
    It suffices to consider functionals of the type:
    \begin{equation*}
      X_{\alpha} X_{\alpha_1} \ldots X_{\alpha_n}
    \end{equation*}
    for $n \ge 1$.
    If $\beta$ is not included in the edges of $\alpha$, we have the property by independence of the associated random variables.

    Let consider the case where $\beta$ is one of the edge of $\alpha$.

\begin{equation*}
  \excond{F}{X^{\{\alpha\}}, \hat X} = \E [g(U_\alpha)]\prod_{b \subset \alpha} \hat X_{b} \prod_{i=1}^n X_{\alpha_i}
\end{equation*}
and
    \begin{equation*}
      \excond{F}{\hat X^{\{\beta\}}, X} = g(U_\alpha) \E \left[\hat X_\beta^{1+\sum_{i=1}^n \ind{\beta \in \alpha_i}} \right]\prod_{i=1}^{n} \prod_{b \subset \alpha_i}\hat X_{b}.
    \end{equation*}
Then,
    \begin{align*}
      \excond{\excond{F}{\hat X^{\{\beta\}}, X}}{X^{\{\alpha\}}, \hat X}
        &= \E [g(U_\alpha)] \E \left[\hat X_\beta^{1+\sum_{i=1}^n \ind{\beta \in \alpha_i}} \right]\prod_{i=1}^{n} \prod_{b \subset \alpha_i}\hat X_{b}.\\
        &= \excond{\excond{F}{X^{\{\alpha\}}, \hat X}}{X, \hat X^{\{\beta\}}}.\\
    \end{align*}
    % where $U_\beta$ is a uniform random variable independent of the rest for $\beta \in A$.
%   \begin{align*}
%     &\int \int F(X_{A \sminus \{a, b\}}, x_a, x_b) \mathbf{P}_a (X_{A \sminus \{a, b\}}, x_b, \dif x_a) \mathbf{P}_b(X_{A \sminus \{a, b\}}, \dif x_b) \\
%     &= \int \int F(X_{A \sminus \{a, b\}}, x_a, x_b) \mathbb{P}^{X_a|Z}(Z, \dif x_a)\dif \delta_{g(Z)} (x_b) \\
%     &= \int \int F(X_{A \sminus \{a, b\}}, x_a, x_b) \dif \delta_{g(Z)} (x_b)\mathbb{P}^{X_a|Z}(Z, \dif x_a) \\
%     &= \excond{\excond{F(X)}{\mathcal{G}^b}}{\mathcal{G}^a}. \\
% \end{align*}
% that yields \eqref{eq-discrete-mall:cond-commut-hatX-X}.
% % \item
% For the second item, if $\alpha, \beta \in A^0$, one has that $\hat X_\alpha$ and $\hat X_\beta$ are dependent of the rando variables $X_\gamma$ for $\gamma \supset \alpha \cup \beta$.
% For $x_\alpha, x_\beta$, one has: $\mathbf{P}_\alpha (X_{A \sminus \alpha, \beta}, x_\beta, \dif x_a) = \mathbb{P}^{X_\alpha| X_\gamma, \gamma \supset \alpha, \hat{X}}$.
  % \end{itemize}
\end{proof}

Those commutation relations of lemma~\ref{lem-discrete-mall:commute-cond-exp} entail a modified Hoeffding decomposition of functionals of Bernoulli random variables.

\begin{lemma}
  \label{lem-discrete-mall:modified-chaos}
  Given $\mathsf X$, the modified chaos decomposition is given by:
  \begin{equation*}
    F = \E [F] + \sum_{n=1}^{+\infty} \pi_n (F)
  \end{equation*}
  with
  \begin{equation}
    \label{eq-discrete-mall:projectors-chaos-cascade}
    \pi_n (F) = \sum_{\substack{I \subset A \cup \hat A \\ |I| = n}} \left(\prod_{b \in I} D_b\right) \left(\prod_{c \in (A \cup \hat A) \sminus I} \E [ \cdot | \mathcal{G}^c] \right)
  \end{equation}
  with $\mathcal{G}^c = \sigma(\mathsf{X}^{\{c\}})$.
\end{lemma}

\begin{proof}[Proof of lemma~\ref{lem-discrete-mall:modified-chaos}]
  We redefine a gradient $D$ and Ornstein-Uhlenbeck operator $\mathsf{L}$ in the same fashion as in subsection~\ref{sec:mall-operators} such that for $a \in A \cup \hat A$:
  \begin{equation*}
    D_a F = F - \excond{F}{\mathsf{X}^{\{a\}}}.
  \end{equation*}
  Then, we follow the same scheme of proof as lemma~\ref{lem-discrete-mall:chaos-decomposition} with that modified gradient.
  Thus, we obtain that $\ker \mathsf{L} = \{F \in \dom{\mathsf{L}} \: : \: \E [F] = 0 \}$ and \eqref{eq-discrete-mall:projectors-chaos-cascade}.

\end{proof}

  The resulting Malliavin framework is analogous to the Malliavin-Dirichlet structure in \cite{decreusefond:hal-01565240}, whose underlying Markov process is the usual Glauber dynamics starting from $\mathsf{X}$. It extends the scope to a particular type of sequences of conditionally independent random variables.
 That applies to $N_G$.
 We recall that in the application to motif estimation, $Z$ is the underlying Erdös-Rényi random graph $\mathbb{G}(n, q_n)$.  The decomposition is similar to \eqref{eq-discrete-mall:N_G-decomposition} except that this time the decomposition involves the random variables $(\hat X_b)_{b \in \hat A}$. Using the inclusion-exclusion principle,
 \begin{align*}
  \excond{N_G}{\mathbb{G}(n, q_n)}  &= \sum_{\substack{H \in \binom{[n]}{3} \\H \simeq G}} \prod_{\alpha \in H} \excond{X_\alpha}{\mathbb{G}(n, q_n)}
  = \sum_{\substack{H \in \binom{[n]}{3} \\H \simeq G}} \prod_{\alpha \in H} \prod_{\beta \subset \alpha} \hat X_\beta \\
  % &= \sum_{\substack{H \in \binom{[n]}{3} \\H \simeq G}} p_n^{|H|}\prod_{\beta \in I^{(2)}} \hat X_\beta \\
  &=  \sum_{\substack{H \in \binom{[n]}{3} \\H \simeq G}} p_n^{|H|} \prod_{\beta \in H^{(2)}}  ((\hat X_\beta - \E [\hat X_\beta]) + \E[\hat X_\beta])  \\
 \end{align*}
 Hence,
 \begin{equation*}
  \excond{N_G}{\mathbb{G}(n, q_n)} - \E [N_G] =  \sum_{\substack{H \in \binom{[n]}{3} \\H \simeq G}} p_n^{|H|}  \sum_{\emptyset \ne J \subseteq H} q_n^{|H^{(2)}| - |J^{(2)}|} \prod_{\beta \in J^{(2)}} (\hat X_\alpha - q_n).
 \end{equation*}
 It entails that:
 \begin{align*}
  N_G - \E [N_G] &= (N_G - \excond{N_G}{\mathbb{G}(n, q_n)}) + (\excond{N_G}{\mathbb{G}(n, q_n)} - \E [N_G]) \\
  &=  \sum_{\substack{H \in \binom{[n]}{3} \\H \simeq G}} \sum_{\emptyset \ne J \subseteq I} p_n^{|H|-|J|} \ind{(H \sminus J)^{(2)} \subset \mathbb{G}(n, q_n)} \prod_{\alpha \in J} (X_\alpha - \excond{X_\alpha}{\mathbb{G}(n, q_n)}) \\
  &+ \sum_{\substack{H \in \binom{[n]}{3} \\H \simeq G}} p_n^{|H|}  \sum_{\emptyset \ne J \subseteq H} q_n^{|H^{(2)}| - |J^{(2)}|} \prod_{\beta \in J^{(2)}} (\hat X_\alpha - q_n) \\
  &= \sum_{\substack{H \in \binom{[n]}{3} \\H \simeq G}} \sum_{\emptyset \ne J \subseteq I} p_n^{|H|-|J|} \prod_{\beta \in (H \sminus J)^{(2)}} \hat X_\beta \prod_{\alpha \in J} (X_\alpha - \excond{X_\alpha}{\mathbb{G}(n, q_n)}) \\
  &+\sum_{\substack{H \in \binom{[n]}{3} \\H \simeq G}} p_n^{|H|}  \sum_{\emptyset \ne J \subseteq H} q_n^{|H^{(2)}| - |J^{(2)}|} \prod_{\beta \in J^{(2)}} (\hat X_\alpha - q_n) \\
  &= \sum_{\substack{H \in \binom{[n]}{3} \\H \simeq G}} \sum_{\emptyset \ne J \subseteq H} p_n^{|H|-|J|} \prod_{\beta \in (H \sminus J)^{(2)}} \hat Y_\beta \prod_{\alpha \in J} Y_\alpha \\
  &+ \sum_{\substack{H \in \binom{[n]}{3} \\H \simeq G}} \sum_{\emptyset \ne J \subseteq H} p_n^{|H|-|J|} q_n^{(H \sminus J)^{(2)}} \prod_{\alpha \in J} Y_\alpha \\
  &+\sum_{\substack{H \in \binom{[n]}{3} \\H \simeq G}} \sum_{\emptyset \ne J \subseteq H} p_n^{|H|}   q_n^{|H^{(2)}| - |J^{(2)}|} \prod_{\beta \in J^{(2)}} \hat Y_\beta
 \end{align*}
 where $Y_\alpha = X_\alpha - \excond{X_\alpha}{\mathbb{G}(n, q_n)}$ and $\hat{Y}_\beta = \hat X_\beta - \E[\hat X_\beta]$. It can be rewritten as $N_G - \E[N_G] = \sum_{J} W^{(1)}_J + W^{(2)}_J + W^{(3)}_J$ where
 \begin{equation}
  \label{eq-discrete-mall:motif-count-condindep-cascade-hoeffding-term}
  \begin{aligned}
 W^{[1]}_{J} =  p_n^{e_G - |J|} q_n^{e^{(2)}_G - |J^{(2)}|} &\left( \sum_{\substack{H \in \binom{[n]}{3} \\H \simeq G, H \supseteq J}} 1 \right) \prod_{\alpha \in J} Y_\alpha ; \quad
    W^{[2]}_J = p_n^{e_G - |J|} q_n^{e^{(2)}_G - |J^{(2)}|} p_n^{|J|} \prod_{\beta \in J^{(2)}} \hat Y_\beta; \\
    W^{[3]}_J &= p_n^{e_G-|J|} \left(\sum_{\substack{H \in \binom{[n]}{3} \\H \simeq G}}  \prod_{\beta \in (H \sminus J)^{(2)}} \hat Y_\beta \right) \prod_{\alpha \in J} Y_\alpha.\\
  \end{aligned}
 \end{equation}
 We consider the Malliavin structure associated to $\mathsf Y = (Y_{\alpha}, \ldots, \hat Y_\beta, \ldots)_{\alpha \in \binom{[n]}{3}, \beta \in \binom{[n]}{3}}$. Then, for each $J$, there exists $m \in \N$ such that $W^{(i)}_J \in \mathfrak{C}_m$ for $i \in \{1, 2, 3\}$.

\begin{theorem}
  \label{thm-discrete-mall:hypergraph-condindep-threshold-hoeffding}
  Let $G$ a hypergraph without isolated vertices. Then, let $p_n \xrightarrow{n \to + \infty} 0$ and $q_n \xrightarrow{n \to + \infty} 0$:
  % \begin{equation}
  %     d_{W}  (\tilde Z(H, \mathbb{T}^{(3)}(n, q_n, p_n)), \Normal(0, 1))  \lesssim \left((1 - p_n q_n) \min \{n^{v_H} p_n^{e_H} q_n^{e^{(2)}_H}\} \right)^{1/2}.
  % \end{equation}
  \begin{equation}
    d_{W}  (\bar N_G, \Normal(0, 1))  \lesssim \left(\min_{\substack{H \subset G \\ e_H > 1}}  \{n^{v_H} p_n^{e_H} q_n^{e^{(2)}_H}\} \right)^{-1/2}.
\end{equation}
% repetition with the above theorem that introduces the notation
% where $e^{(2)}_H$ is the number of edges that are subsets of the hyperedges in $G$.
\end{theorem}

\begin{proof}[Proof of theorem~\ref{thm-discrete-mall:hypergraph-condindep-threshold-hoeffding}]
  We follow the same lines as the proof of theorem~\ref{thm-discrete-mall:hypergraph-condindep-threshold}, with the difference that $\pi_0(N_G) = \E [F]$.
  The \eqref{eq-discrete-mall:eigenfunction} assumption holds.
  We recall the bound in our context
  \begin{equation}
    d_{W} (\bar N_G, \Normal (0, 1)) \le C_{e_G} \sqrt{\sum_{(I, J, K, L) \textnormal{ connected}} \sum_{i_i, i_j, i_k, i_l = 1}^{3} | \E [W_I^{[i_i]} W_J^{[i_j]} W_K^{[i_k]} W_L^{[i_l]}] | / \var [N_G] }.
  \end{equation}
  As each connected quadruple $(I, J, K, L)$ is associated to $(H_1, H_2, H_3, H_4)$ subhypergraphs of $G$ such that $I \simeq H_1$, $J  \simeq  H_2$, $K  \simeq H_3$ and $L \simeq H_4$, from theorem~\ref{thm-discrete-mall:hypergraph-condindep-threshold}, we have:
\begin{multline}
  \label{eq-discrete-mall:connected-W_I_1234_upper_bound}
  |\E [W_I^{[1]} W_J^{[i_j]} W_K^{[i_K]} W_L^{[i_l]}]|  \\ \le w_I w_J w_K w_L \, n^{v(H_1 \cup H_2 \cup H_3 \cup H_4)} p_n^{|H_1 \cup H_2 \cup H_3 \cup H_4|} q_n^{|H_1^{(2)} \cup H_2^{(2)} \cup H_3^{(2)} \cup H_4^{(2)}|}
\end{multline}
for $i_j, i_k,i_l \in \{1, 2\}$ as $\prod_{\alpha \in J}\excond{Y_\alpha}{Z} \propto_Z \prod_{\beta \in J^{(2)}} \hat Y_\beta$ and $\E [\hat Y_\beta^k] \propto p_n$ for $k \ge 2$.
As $\hat X_\beta \le 1$ a.s., we also have \eqref{eq-discrete-mall:connected-W_I_1234_upper_bound} for all $i_i, i_j, i_k, i_l \in \{1, 2, 3, 4\}$.
Likewise, the variance reads off in function of the quadruples:
\begin{equation*}
  \var [N_G] = \frac{1}{4} \sum_{\substack{H_1, H_2 \subset G \\ H_1 \cap H_2 \neq \emptyset}}  \left(\sum_{I}^{*H_1}\E [W_I^2] + \sum_{J}^{*H_2}\E [W_J^2]\right)
\end{equation*}
where $\sum_{I}^{*H_1} \cdots$ stands for a sum over $I$ such that $I$ is isomorphic to $H_1$.

We follow the same lines of computations as those leading to \eqref{eq-discrete-mall:var-square}. Then, for fixed quadruples $(H_1, H_2, H_3, H_4)$,
\begin{equation}
  \var [N_G] \ge \frac{1}{16}  \left(\sum_{I}^{*H_1}\E [W_I^2] \times \sum_{J}^{*H_2}\E [W_J^2] \times \sum_{K}^{*H_3}\E [W_K^2] \times \sum_{L}^{*H_4}\E [W_L^2]\right)^{1/2}.
\end{equation}
As
$\E [W_I^2] = p_n^{|H| - |H_1|} q_n^{|H^{(2)}| - |H_1^{(2)}|} q_n^{|H_1^{(2)}|}(1 - q_n)^{|H_1^{(2)}|}(1 - p_n)^{|H_1|}p_n^{|H_1|},$ we have:

\begin{multline*}
  \sum_{(I, J, K, L) \textnormal{ connected}} \sum_{i_i, i_j, i_k, i_l = 1}^{3} | \E [W_I^{[i_i]} W_J^{[i_j]} W_K^{[i_k]} W_L^{[i_l]}] | / \var [N_G] \\
  \le \frac{n^{v(H_1 \cup H_2 \cup H_3 \cup H_4)} p_n^{|H_1 \cup H_2 \cup H_3 \cup H_4|} q_n^{|H_1^{(2)} \cup H_2^{(2)} \cup H_3^{(2)} \cup H_4^{(2)}|}}{\left(n^{v(H_1) + v(H_2) + v(H_3) + v(H_4)} p_n^{|H_1| + |H_2| + |H_3| + |H_4|} q_n^{|H_1^{(2)}| + |H_2^{(2)}| + |H_3^{(2)}| + |H_4^{(2)}|}\right)^{1/2}}.
\end{multline*}
At that point, we arrive at the same upper bound as in the proof of theorem~\ref{thm-discrete-mall:hypergraph-condindep-threshold}.
\end{proof}
While in \cite{Kaur2021,Temcinas2022a}, the probability of keeping a hyperedge does not depend on the number of vertices, we let $p_n$ tend to 0. As a consequence, we can state thresholds for subhypergraph containment that complement the ones in \cite[p.61]{janson2000random}.
As done in \cite{Kaur2021,zhang2022berry} for random graphs, it should be possible to derive with our method the convergence rates considering an arbitrary exchangeable random hypergraph generated by a hypergraphon, the analog of graphon in graph limit theory.
% Example of new applications
% \chris{Another idea is to deal with subhypergraph counts in hypergraph models with random hyperedge probabilities as it is done for Poisson approximation in \cite{coulson2016poisson}. }
% \todo[]{Mild assumption on the latent variable}
% We can basically make any random variables latent

% \todo[]{not true, can only deal with one layer of random object}
% Another class of estimation problem on random graphs is to formulate a percolation problem. In the same way as \cite{krokowski2017discrete}, one can use our method to extend that kind of results to random hypergraphs.
%
% \todo[]{Comment on Röllin and multivariate for sake of completeness}
% \begin{remark}
%   In \cite{Kaur2021}, the authors apply Stein's method for \chris{multivariate functionals} of conditionally independent random variables, because the summands are uncorrelated. In the case, where $p_n \to 0$ , there are correlations between them.
% \end{remark}
% \chris{The generalization of our result to multivariate target distribution will be considered in a forthcoming paper.}

%%% Local Variables:
%%% mode: latex
%%% TeX-master: discrete-mall
%%% End:

% Bibliography
% \bibliographystyle{amsplain} 
%\bibliographystyle{plain}
%\bibliography{references}

\end{document}